\documentclass[a4paper, 10pt]{amsart}

\usepackage[square, numbers, comma]{natbib} 

\usepackage[usenames,dvipsnames,svgnames,table]{xcolor}
\usepackage{amsmath,amsthm,amssymb,amssymb,esint,verbatim,tabularx,graphicx}
\usepackage{bbm}
\usepackage{fancyhdr}
\usepackage{enumerate}
\usepackage{amsmath}
\usepackage{fancyhdr}
\usepackage{epic}
\usepackage{pgf,tikz}
\usetikzlibrary{arrows}

\usepackage[utf8]{inputenc}
\usepackage{color}
\usepackage{hyperref}
\usepackage{verbatim}
\usepackage{pdfpages}
\usepackage{empheq}

\hypersetup{urlcolor=blue, colorlinks=true} 

\usepackage[inner=2.5cm,outer=2.5cm,bottom=4cm,top=4cm]{geometry}

\newtheorem{theorem}{Theorem}[section]

\newtheorem{lemma}[theorem]{Lemma}

\newtheorem{corollary}[theorem]{Corollary}
\newtheorem{definition}{Definition}[section]

\theoremstyle{remark}
\newtheorem{remark}[theorem]{Remark}
\theoremstyle{definition}

\numberwithin{equation}{section}

\newcommand{\R}{\ensuremath{\mathbb{R}}}
\newcommand{\N}{\ensuremath{\mathbb{N}}}

\newcommand{\Z}{\ensuremath{\mathbb{Z}}}

\newcommand{\veps}{\varepsilon}

\newcommand{\s}{\ensuremath{s}}

\newcommand{\flap}{\ensuremath{(-\Delta)^{\s}}}
\newcommand{\iflap}{\ensuremath{(-\Delta)^{-\sigma}}}

\newcommand{\LL}{\mathcal{L}}

\newcommand{\dd}{\,\mathrm{d}}

\DeclareMathOperator{\supp}{supp}

\newcommand{\Se}{\mathcal{S}_{\tau}}

\begin{document}

\title[Finite differences for the porous medium with nonlocal pressure]{A convergent finite difference-quadrature scheme for the porous medium equation with nonlocal pressure}

\author[F.~del~Teso]{F\'elix del Teso}

\address[F. del Teso]{Departamento de Matem\'aticas, Universidad Aut\'onoma de Madrid (UAM), Campus de Cantoblanco, 28049 Madrid, Spain} 

\email[]{felix.delteso\@@{}uam.es}

\urladdr{https://sites.google.com/view/felixdelteso}


\author[E.~R.~Jakobsen]{Espen R. Jakobsen\\ \\ \\ {\scriptsize Communicated by Endre Süli}}

\address[E.~R.~Jakobsen]{Department of Mathematical Sciences, Norwegian University of Science and Technology}
\email[]{espen.jakobsen\@@{}ntnu.no}
\urladdr{https://folk.ntnu.no/erj/index.html}

\address[]{}
\email[]{}
\urladdr{}


\date{17 December 2025: Accepted for publication in the journal Foundations of Computational Mathematics.}

\begin{abstract}\noindent   
  We introduce and analyze a numerical approximation of the
  porous medium equation with fractional potential pressure introduced
  by Caffarelli and V\'azquez:
\[
\partial_t u = \nabla \cdot (u^{m-1}\nabla (-\Delta)^{-\sigma}u) \qquad \textup{for} \qquad m\geq2 \quad \textup{and} \quad \sigma\in(0,1).
\]
Our scheme is for one space dimension and positive solutions $u$. It
consists of solving numerically the equation satisfied by
$v(x,t)=\int_{-\infty}^xu(y,t)dy$, the quasilinear nondivergence form equation
\[
\partial_t v= -|\partial_x v|^{m-1} (- \Delta)^{s} v \qquad  \textup{where} \qquad  s=1-\sigma,
\]
and then computing $u=v_x$ by numerical differentiation. Using
upwinding ideas in a novel way, we construct a new and simple, monotone and
$L^\infty$-stable, approximation for the $v$-equation.  The
full scheme then becomes a conservative up-wind finite volume 
approximation for the $u$-equation. We show local uniform
convergence to the unique discontinuous viscosity solution for the
$v$-problem, and using ideas from probability
theory, we prove that the 
approximation of $u$ converges 
up to normalization in $C(0,T; P(\R))$ where $P(\R)$ is the space of probability
measures under the Rubinstein-Kantorovich  (bounded Lipschitz)  metric.  
The analysis includes
also fundamental solutions where the initial data for $u$ is a Dirac mass. Numerical tests are
included to support the results. Our scheme seems to
be the first numerical scheme for this type of problems.
\end{abstract}

\maketitle

\tableofcontents 

\addtocontents{toc}{\protect\setcounter{tocdepth}{1}}

\section{Introduction} \label{sec:intro}

In this paper we introduce and analyze 
the first numerical approximation of the
fractional pressure equation studied by Caffarelli and V\'azquez in \cite{CaVa11} for $m=2$:
\begin{align}
\partial_t u - \nabla \cdot (u^{m-1} \nabla \iflap u)&=0 \qquad \textup{in} \quad \R^d\times(0,T), \label{eq:maineq} \tag{FPE} \\
u(\cdot,0)  &=\mu_0  \hspace{5.09mm}   \textup{in} \quad \R^d  \label{eq:maineqini}  \tag{IC$_u$},
\end{align} 
 where  
$m>1$, $\sigma\in (0,1)$, 
and $\mu_0\in \mathcal{M}^+(\R^d)$, the
space of nonnegative finite Radon measures. The operator
$\nabla
\iflap=\nabla^{1-2\sigma}$ is a fractional gradient of order $1-2\sigma$
defined in Section \ref{sec:theoresults2}. This equation was first
introduced for $m=2$ in  \cite{BiKaMo10} as the derivative
of a nonlinear nonlocal dislocation model. In \cite{CaVa11}
 it was studied as a nonlinear porous medium flow with
fractional potential pressure, and then extended to the full slow
diffusion range $m\geq2$ in \cite{StTeVa16}.
It is one out of the two main nonlocal generalizations of the
porous medium equation (PME). The other one,
\begin{align}\label{2ndFPME}
\partial_t u + (-\Delta)^{s}u^m &=0 \qquad \textup{in} \quad \R^d\times(0,T), 
\end{align} 
along with many extensions of the form $\partial_t u - \mathcal L
\varphi(u) =f$, have been considered  e.g. in \cite{dP-etal11,dP-etal12,CiJa11,dTEnJa17}.  While both equations preserve mass 
and energy, some other properties of the local PME are lost. As
opposed to equation \eqref{2ndFPME}, solutions of equation
\eqref{eq:maineq}
has finite speed of propagation as in the local case \cite{CaVa11,StTeVa16}.
However, comparison is lacking and uniqueness is challenging
\cite{CaVa11,StTeVa16}. Equation  \eqref{2ndFPME} has comparison and general uniqueness (even in the fast diffusion range
$m\in(0,1)$) 
\cite{dP-etal11,dP-etal12,dTEnJa17}. 
Existence of weak solutions of \eqref{eq:maineq} holds in any
dimension \cite{CaVa11,StTeVa16}, while uniqueness is only known in
dimension $d=1$ where  it is obtained indirectly via a  related integrated problem: 
\begin{align}
\partial_t v + |\partial_x v|^{m-1} \flap v&=0,\qquad \textup{in} \quad \R\times(0,T)=:Q_T, \label{eq:integP} \tag{IP} \\
u(\cdot,0)  &=v_0,  \hspace{5.5mm}   \textup{in} \quad
\R,  \label{eq:integPini}  \tag{IC$_v$} 
\end{align}
 where  $s=1-\sigma$. 
When $v_0(x)=\int_{-\infty}^x \dd \mu_0(y)$, the solutions of the two
problems satisfy
$$v(x,t)=\int_{-\infty}^x u(y,t)\dd y\qquad\text{and}\qquad v_x(x,t)=u(x,t).$$
Equation \eqref{eq:integP} is a degenerate quasilinear parabolic equation in
nondivergence form. Solutions need not to be smooth and well-posedness
is understood in the viscosity solution sense
\cite{BiKaMo10,ChJa17}. 
The relation between
\eqref{eq:maineq} and \eqref{eq:integP} was noted in
\cite{BiKaMo10} for $m=2$  while studying dislocation models, and the currently most general uniqueness result for \eqref{eq:maineq} is
given in  \cite{StTeVa16}, see also \cite{ChJa17}. 
Note that when $u$ is nonnegative, then $v$ must be
nondecreasing. This will be the setting of this paper.

While different numerical approximations have been proposed for
\eqref{2ndFPME} and its various extensions \cite{Dro10,CiJa14,DrJa14,dTEnJa18,dTEnJa19},  we are not aware of any approximation schemes for
\eqref{eq:maineq}. The goal of this paper is to introduce a robust and
stable numerical method for \eqref{eq:maineq} for which we can prove
convergence. 

This method will be fundamentally different from the
discretizations of \eqref{2ndFPME}, since it will be based on the integrated problem
\eqref{eq:integP}-\eqref{eq:integPini}.
The method we propose has two steps. First we construct 
a  conditionally  consistent\footnote{ Consistent for nondecreasing solutions where $\max\{D_hV_i^j,0\}^{m-1}=|D_hV_i^j|^{m-1}=(D_hV_i^j)^{m-1}$.}, stable, monotone,
and explicit in time discretization for \eqref{eq:integP}-\eqref{eq:integPini}:
\begin{align}\tag{S}
  V_i^{j+1}&=V_i^j + \tau 
  \max\{D_hV_i^j,0\}^{m-1}
  \flap_h  V_i^{j}, && i\in \Z, \quad j\in \N, \label{schemeVintro}\\
  V_i^0 &=v_0(x_i),&& i\in\Z,\tag{IC$_S$} \label{schemeVintroini}
\end{align}
where $V_i^j\approx v(x_i,t_j)$ for all grid points $(x_i,t_j)$, 
and $D_h$ and $\flap_h$ are defined in Section \ref{sec:main}. 
Then from
the solution $V_i^j$ we compute an approximation $U^j_i$ of the solution
$u(x_i,t_j)$ of \eqref{eq:maineq}-\eqref{eq:maineqini} by numerical
differentiation:  
\begin{equation}\label{SchemeUintro}
U^j_i=\frac{V^j_{i+1}-V^j_i}h.
\end{equation}
The discretization we propose for $V$ is a combination of an explicit Euler
approximation in time and a novel finite difference quadrature
approximation of the quasilinear operator. One of the main
contributions of this paper is the construction of this discretization
so that it becomes monotone. We do this using upwinding ideas in a way 
we have never seen before. The resulting discretization has a compact
stencil and is much simpler and
more direct than other monotone approximation of quasilinear operators
we know of. Typically such discretizations rely on extra 
discretization parameters, wide stencils, and/or linearization and fixed point iterations:
 See e.g. \cite{CrLi96,De00,Ob04} for mean curvature and similar equations, 
\cite{LoRa05,Ob08,FeNe09,FrOb11,BeFrOb14,NoNtZh19} for Monge-Ampere equations, \cite{dTLi22,dTMeOc23} for $p$ and $\infty$-Laplacian and fractional $p$-Laplacian equations, and \cite{Obe05,Obe13,BuCaRo22a,LiSa22,BuCaRo22} for normalized versions of these operators. 

The scheme given by
\eqref{schemeVintro}-\eqref{schemeVintroini} and \eqref{SchemeUintro}
can be interpreted as a finite volume approximation of
\eqref{eq:maineq}. When $m=2$ it looks like a classical 
conservative upwind scheme \cite[Example
  1.1]{EGH00}:
\[
U_i^{j+1}=U_i^j - \tau
\frac{F_{i+\frac{1}{2}}^j-F_{i-\frac{1}{2}}^j}{h}\qquad \textup{for} \qquad F_{i+\frac{1}{2}}^j= U_{i}^j \big(\Phi_{i+\frac{1}{2}}^j\big)_+  + U_{i+1}^j \big(\Phi_{i+\frac{1}{2}}^j\big)_{-},
\]
where $\Phi$ is an approximation of \(
\phi= - \nabla^{1-2\sigma} u=   (-\Delta)^{s} v\) for
$s=1-\sigma$. E.g. when
we discretize $(-\Delta)^s$ by \emph{powers of the
  discrete Laplacian} $(-\Delta_h)^s$ (see \cite{Cia-etal18} and
Appendix \ref{app:disfl}), then
\[
\Phi_{i+\frac{1}{2}}^j= (-\Delta_h)^s V_{i}^j = -D_h^+  (-\Delta_h)^{-\sigma} U_i^{j}
\qquad \text{where}\qquad \text{$D^{\pm}_h$ are forward/backward differences.}\]
We refer to Section \ref{sec:FVM} for more details. This type of
schemes have many similarities to schemes for other types of 
diffusion-convection-aggregation problems
\cite{EGH00,CP04,BF12,CCH15,BCMS20}.
Here a major difficulty is to design the schemes so that they are provably
convergent. Previous convergence analyses mostly follow the setup of
\cite{EGH00}
exploiting uniform $L^\infty$ and energy estimates for the problem in
the original variables. In some 
cases convergence can be proved even for degenerate 
equations \cite{EGH00,CP04}, while in other cases which are somewhat closer to
our setup, convergence holds in nondegenerate regions
\cite{BCMS20}.  We refer to \cite[Section 7]{Gom23} for a recent survey on the topic in the context of aggregation-diffusion equations.

The main results of this paper are convergence results as $h,\tau\to
0$ under CFL-conditions.  For suitably defined interpolants of $U$
and $V$, denoted by $\bar U$ and $\bar V$ respectively, local uniform convergence of $\bar V\to
v$ in time and space and weak-$*$ convergence in $C_b(\R)$ of $\bar U\to u$
 locally uniformly  in time. We cover both ``regular'' cases when $0\leq u_0\in
L^1(\R)$ and ``singular'' cases when $u_0$ is a delta measure. In the
latter case $u$ is a fundamental solution and the solution $v$ has
discontinuous initial data. The situation is reminiscent of
probability theory where pointwise convergence of cumulative
densities (``$v$'') implies weak convergence of density functions
(``$u$''), and another contribution of this paper is the use
of such ideas in numerical analysis to prove convergence of $U$ from
the local uniform convergence of $V$.  The convergence argument for
$U$ also needs tightness for $u$ and $U$, which we prove in an
untraditional way exploiting the relation between $(u,U)$ and $(v,V)$, 
mixing mass preservation and finite speed of propagation for
\eqref{eq:maineq},  
comparison for \eqref{eq:integP}, and the convergence of $\bar V\to
v$. To show convergence of $V$, we prove a 
number of a priori, monotonicity, and asymptotic results for $V$,
consistency of the scheme, and conclude using the so-called ``half-relaxed
limit'' method of Barles-Perthame-Souganidis \cite{BaSo91}. Most of
these results crucially rely on the discretization being monotone. Note that
in the ``singular'' data case, the convergence is locally uniform except
at the discontinuity point and the limit has to be interpreted as a
discontinuous viscosity solution. In this paper we extend this type of
theory to our setting and in some cases give more precise results than
previously. 

 In the end we formulate our convergence results for $U$ in terms of
the Rubinstein-Kantorovich  (or bounded-Lipschitz)  distance $d_0$, which
metrizes weak-$*$ convergence in the space of probability measures $P(\R)$. The
result, $\max_{t\in[0,T]}d_0(u(t),\bar U(t))\to0$, is equivalent (up to
normalization) to convergence in $C([0,T];P(\R))$.  
Such type of (weak-$*$/Wasserstein-like) convergence results for numerical
schemes are not completely new. In the setting of hyperbolic conservation laws
there are several such results e.g. in
\cite{NeTa92,FjSo16,Fjo-etall17,RuSaSo19}. In these papers,
convergence is proved using duality arguments which are different from
our ``cummulative density'' approach. We also refer to \cite{CaFjSo21}
for a 1-Wasserstein convergent approximation of the aggregation equation.
 
Finally we mention that the finite volume interpretation of our
approximation of \eqref{eq:maineq} allows for 
generalizations to more general equations and higher
dimensions. However, since the interpretation in terms of cumulative 
densities is then lost, a different approach is needed to prove
convergence.
We can also handle the fast diffusion range of \eqref{eq:maineq} where
$m\in(1,2)$ and the model has infinite speed of propagation. This case can be
treated as in \cite[Section 
  4.6]{dTEnJa18} by approximating $u^{m-1}$ by $f_{\veps}(u)=(u +
\veps)^{m-1}-\veps^{m-1}$ in \eqref{eq:maineq}. 
The CFL conditions will then depend also on the extra
approximation parameter $\veps$.
\smallskip

\emph{ Discussion of literature:} We give some more details for the model
\eqref{eq:maineq}. In \cite{CaVa11} the authors proved existence of
weak solutions, finite speed of 
propagation, lack of comparison principle, and a series of energy
estimates. Asymptotic time behavior of the solutions were analyzed in
\cite{CaVa11b}, $L^1$-$L^\infty$ smoothing effect and regularity of
solutions when $s\not=1/2$ in \cite{CaSoVa13} (with Soria) and when
$s=1/2$ in \cite{CaVa15}. The explicit 
form of a self-similar solution was found by Biler, Imbert and Karch in
\cite{BiImKar11,BiImKa15}. The model was then extended to the full slow
diffusion range $m>1$ in
Stan, del Teso, and V\'azquez \cite{StTeVa16}. Here the
authors proved finite speed of 
propagation for $m\geq2$, infinite speed of propagation for
$m\in(1,2)$,
$L^1$-$L^\infty$ smoothing, existence for measure data, while asymptotic
behaviour in dimension $d=1$ is given in \cite{StTeVa18b}.  The
model has also been studied as a Wasserstein gradient flow of a
fractional Sobolev norm when $m=2$ in \cite{LiMaSe18},  and for more general nonlocal interactions in \cite{FoPaSc25}. Finally, we mention the recent preprint \cite{CaFrSu24} where a  finite element method is developed for a nondegenerate  version of \eqref{eq:maineq} obtained by adding a 
linear diffusion term.

A scheme for an integrated equation is given in \cite{CaFCVa22}. Here this problem corresponds the simpler limit case $s=0$ ($\sigma=1$)  where the resulting equation is a relatively standard local first order Hamilton-Jacobi equation. Thus, classical monotone schemes from the literature can be used. However, there is no convergence analysis for the original problem.  In the limit case $\sigma=0$ we recover the standard Porous Medium Equations. Most of the arguments of this paper also applies in this case, but the resulting scheme is indeed more complicated than a  finite difference scheme of the original problem (see e.g. \cite{dTEnJa19}).  
 
The model \eqref{eq:integP} was studied first in the case $m=2$ by
Biler, Karch, Monneau \cite{BiKaMo10}, and later for any $m>1$, as
special case of a much more general theory by Chasseigne and Jakobsen in
\cite{ChJa17} 
that includes nonlocal versions of $p$-Laplace evolution
equations. Both establish comparison, existence, and uniqueness
results for viscosity solutions, as well as
nonlocal to local limits of these problems in \cite{ChJa17}. 

The rest of this paper is organized as follows: Section \ref{sec:theoresults2}
contains the definitions of fractional gradients and weak solutions,
along with the well-posedness, stability, and tightness 
results for the two equations. In Section \ref{sec:main} we construct the
numerical schemes and state the main convergence results. Then in
Section \ref{sec:theoresults} we define (discontinuous) viscosity solutions for
\eqref{eq:integP} and state and prove several results that we will
need in the later analysis. The discretizations of \eqref{eq:integP}
and \eqref{eq:maineq} are analyzed in Sections \ref{sec:propNumSch}
and \ref{sec:propUandconv} respectively, with a priori estimates and
convergence. In section \ref{sec:num} we do numerical experiments to
check the convergence of our schemes. Finally, in the appendix we
give all the details of one of the many monotone discretizations of the
fractional Laplacian for which our results hold.

\section{Preliminaries}
  \label{sec:theoresults2}

In this section we define precisely the nonlocal gradient and give
well-posedness result for the two problems \eqref{eq:maineq}-\eqref{eq:maineqini} and \eqref{eq:integP}-\eqref{eq:integPini}.
For $s\in(0,1)$, the fractional Laplacian $\flap$ is defined as a
Fourier multiplier with symbol $\widehat\flap=|\xi|^{2s}$, or equivalently as a singular integral
\[
\flap\phi(x)=c_{d,s} \textup{P.V.} \int_{|y|>0}  (\phi(x)-\phi(x+y)) \frac{\dd y}{|y|^{d+2s}},
\]
for smooth functions $\phi$  and $c_{d,s}=\frac{2^{2s}
  \Gamma(\frac{d+2s}{2})}{\pi^{\frac{d}{2} }|\Gamma(-s)|}$. The \emph{nonlocal gradient}
  $\nabla^{1-2\sigma}$ (see \cite{BiImKa15}), is defined to the
Fourier multiplyer with symbol $\widehat{\nabla^{1-2\sigma}}= 
i\xi|\xi|^{2\sigma}$, or in the singular integral representation as 
\[
\nabla^{1-2\sigma} \phi(x)= \tilde{c}_{d,s} \int_{\R^d} (\phi(x)-\phi(x+y))\frac{y}{|y|}\frac{\dd y}{|y|^{d+1-2\sigma}}.
\]
This operator is well-defined for $\sigma\in(0,1)$, and formally
$\nabla^{1-2\sigma}=\nabla \iflap$. The latter representation is rigorous
 when $\sigma \in
(0,1)$ and $d>2\sigma$, which is the range where $(-\Delta)^{\sigma}$ has an
inverse.\footnote{The inverse is then the Riesz potential of order $2\sigma$, 
\[
\iflap \phi(x)= c_{d,-\sigma}\int_{\R^d} \phi(x-y) \frac{\dd y}{|y|^{d-2\sigma}}.
\]}
It is standard to check that $\nabla \cdot\nabla^{1-2\sigma} =
(-\Delta)^{1-\sigma}$. 

We introduce now the concept of weak solutions for \eqref{eq:maineq}-\eqref{eq:maineqini}.

\begin{definition}
We say that $u:\R^d\to \R_+$ is a weak solution of \eqref{eq:maineq}-\eqref{eq:maineqini} with initial data $\mu_0\in \mathcal{M}^{+}(\R^d)$ if $u\in L^1_{\textup{loc}}(\R^d\times(0,T))$, $u^{m-1}\nabla \iflap u \in L^1_{\textup{loc}}(\R^d\times(0,T))$, and 
\begin{equation*}
\int_0^T\int_{\R^d} u(x,t) \phi_t(x,t)\dd x\dd t-\int_0^T\int_{\R^d}  u^{m-1}(x,t) \nabla (-\Delta)^{-\sigma} u(x,t) \cdot \nabla \phi(x,t) \,\dd x\dd t+  \int_{\R^d} \phi(x,0) \dd \mu_0(x)=0,
\end{equation*}
for all test functions  $\phi \in C^1_c(\R^d \times [0,T))$.
\end{definition}

The following theorem provides existence and some properties of weak
solutions of \eqref{eq:maineq}-\eqref{eq:maineqini} for different
types of initial data. A proof and more qualitative properties of
solutions can be found in \cite{StdTVa18}. 

\begin{theorem}\label{thm:maineq}
Assume  $m\in (1,+\infty)$,  $\sigma\in (0,1)$,  $d\ge 1$ and $\mu_0 \in \mathcal{M}^{+} (\R^d)$. 
\begin{enumerate}[(a)]
\item\label{thm:maineq-itemA} Then there exists a nonnegative weak solution $u$ of problem
  \eqref{eq:maineq}-\eqref{eq:maineqini} such that,
  \smallskip
  \begin{align*}
    &u \in  L^\infty((0,\infty);L^1(\R^d))\cap L^\infty (\R^d \times
    (\eta, \infty)) \qquad \forall\eta>0,\\[0.2cm]
&\|u(\cdot,t)\|_{L^1(\R^d)}=\mu_0(\R^d)\qquad \forall t>0,\qquad  \text{and} \\[0.2cm]
&\| u(\cdot,t)\|_{L^{\infty}(\R^d)} \le C_{N,s,m} \, t^{-\gamma}
\mu_0(\R^d)^{\delta}\qquad\text{for}\qquad
\gamma=\tfrac{d}{(m-1)d+2(1-s)},\quad \delta=\tfrac{2(1-s)}{(m-1)d+2(1-s)}.
  \end{align*}
  \smallskip
\item If  $\mu_0=u_0\in L^1(\R^d)$, 
then 
$u \in  L^\infty( [  0,\infty);L^1(\R^d))\cap L^\infty (\R^d \times
  (\tau, \infty))$ for all $\tau>0$, and
  $$\|u(\cdot,t)\|_{L^1(\R^d)}=\|u_0\|_{L^1(\R^d)}.$$
\item If $\mu_0=u_0\in L^1(\R^d)\cap L^\infty(\R^d)$, then $u \in  L^\infty( [  0,\infty);L^1(\R^d)\cap L^\infty(\R^d))$ and 
  $$\|u(\cdot, t)\|_{L^\infty(\R^d)}\leq \|u_0\|_{L^\infty(\R^d)}.$$
  \item  If the dimension $d=1$, then $u$ is unique in
  $L^\infty((0,\infty);L^1(\R)).$
  \medskip
  \item\label{thm:maineq-itemFSP} If $m\geq2$, then $u$ has finite speed of propagation in the
    sense that, if $\supp\{\mu_0\}$ is a bounded set then
    $\supp\{u(\cdot,t)\}$ is also bounded for all $t\geq0$.
    \medskip
    \item\label{thm:maineq-itemtight} If $m\geq2$ and the dimension $d=1$, then $u(t,\cdot)$ is tight: For every $\veps,T>0$ there exists $R=R(\veps,\mu_0,T)>0$
      such that 
\[
\sup_{t\in[0,T]} \int_{|x|>R} u(x,t)\dd x<\veps.
\]
\end{enumerate}
\end{theorem}

Existence and properties (a), (b), (c) and (e) follow from
\cite{StdTVa18}. Uniqueness in dimension $d=1$ (d) can be found in
\cite{StTeVa18b}. We prove (f) at the end of the this section.
At this point we also mention that weak solutions of
\eqref{eq:maineq}-\eqref{eq:maineqini} have  
infinite speed
  of propagation when $m\in(1,2)$ in dimension $d=1$ \cite{StTeVa16,StdTVa18}.  When the dimension $d=1$ and $u_0=M\delta_0$ for $M>0$, the solution is
  self-similar and determines the asymptotic behaviour of weak
  solutions of \eqref{eq:maineq}-\eqref{eq:maineqini} for arbitrary initial
  data with mass $M$ \cite{StTeVa18b}.

From these results on \eqref{eq:maineq}-\eqref{eq:maineqini}, we
obtain results for \eqref{eq:integP}-\eqref{eq:integPini} in dimension
$d=1$.  Define
\begin{align}\label{def-v}
v_0(x)= \int_{-\infty}^x u_0(y)\dd y \qquad \textup{and} \qquad v(x,t)= \int_{-\infty}^x u(y,t)\dd y \qquad \textup{for all} \qquad t\geq0.
\end{align}
We will show that $v$ solves \eqref{eq:integP}-\eqref{eq:integPini}
 with $s=1-\sigma$  in the sense of viscosity solutions. 
 Definitions of viscosity and discontinuous viscosity solutions are given in 
 Section \ref{sec:theoresults}. 
First we present results for bounded integrable $u_0$. We
sketch some proofs for completeness and refer to the original sources
for others. 

We recall the notation $Q_T:=\R\times(0,T)$.  

\begin{theorem}\label{thm:l1linfv}
Assume $m\in (1,+\infty)$,  $\sigma\in (0,1)$, $s=1-\sigma$,  $d=1$, $0\leq u_0\in
L^1(\R)$, and $u$ is the weak solution of
\eqref{eq:maineq}-\eqref{eq:maineqini} given by Theorem
\ref{thm:maineq}.  If $v$ is the function defined in \eqref{def-v},
then $v\in BUC(\overline{Q}_T)$ is the unique viscosity solution of
\eqref{eq:integP}-\eqref{eq:integPini} and satisfies:
\medskip
\begin{enumerate}[(i)]
\item[(i)\ ] $0\leq v \leq
  \|u_0\|_{L^1(\R)}$ in $\overline{Q}_T$.
  \medskip
\item[(ii)\,] $v(\cdot,t)$ is nondecreasing for $t\geq0$,
  \medskip
\item[(iii)]
\(|v(x_2,t)-v(x_1,t)|\leq C_{1,s,m} \, t^{-\gamma}
\|u_0\|_{L^1(\R)}^{\delta} |x_2-x_1|\),\qquad for\quad $t>0$\quad
and\quad $x_1,x_2\in\R$, 
\medskip
\item[] where $\gamma,\delta>0$ are defined in Theorem \ref{thm:maineq}.
\end{enumerate}
\smallskip
If in addition $u_0\in
L^1(\R)\cap L^\infty(\R)$, then
\smallskip
\begin{enumerate}
\item[(iii)']
\(
|v(x_2,t)-v(x_1,t)|\leq \|u_0\|_{L^\infty(\R)}|x_2-x_1|
\),\qquad for\quad $t>0$\quad and\quad $x_1,x_2\in\R$.
\medskip
\end{enumerate}
 Moreover, the \eqref{eq:integP}-\eqref{eq:integPini} has a comparison principle for sub and super solutions.

\end{theorem}

Part (iii) and (iii)' means that $v$ is $x$-Lipschitz, in the latter
case with a bounded in $t$ Lipschitz constant.
We refer to Section \ref{sec:theoresults} for the 
definitions and basic properties and results for continuous and
discontinuous viscosity solutions.

\begin{proof}
We start by proving the properties of $v$. Since $u\geq0$, then
trivially $v(\cdot,t)$ is nondecreasing (ii). Moreover (i) follows,
since for any $(x,t) \in \overline{Q}_T$ we have
\[
0\leq 
\int_{-\infty}^xu(y,t)\dd y =v(x,t)\leq \lim_{x\to\infty}\int_{-\infty}^xu(y,t)\dd y=\|u(\cdot,t)\|_{L^1(\R)}=\|u_0\|_{L^1(\R)}.
\]
Next, for $x_1,x_2\in \R$ such that $x_2\geq x_1$, we have that
\[
|v(x_2,t)-v(x_1,t)|= \int_{-\infty}^{x_2}u(y,t)\dd y - \int_{-\infty}^{x_1}u(y,t)\dd y = \int_{x_1}^{x_2}u(y,t)\dd y \leq \|u(\cdot,t)\|_{L^\infty(\R)}|x_2-x_1|.
\]
By Theorem \ref{thm:maineq} (a), $\|u(\cdot,t)\|_{L^1}\leq
\|u_0\|_{L^1}$, so $v$ is uniformly continuous in $x$ uniformly in
time by the dominated convergence theorem. 
Moreover, by Theorem \ref{thm:maineq} (a)-(c), 
$\|u(\cdot,t)\|_{L^\infty}\leq \|u_0\|_{L^\infty}$ when
$u_0\in L^\infty$, and $\|u(\cdot,t)\|_{L^\infty}\leq C_{N,s,m} \, t^{-\gamma}
\|u_0\|_{L^1}^{\delta}$ if not. Hence parts (iii) and (iii)'
follow. Continuity in time follows as in Theorem 8.2 in
\cite{StTeVa16}.
The fact that $v$ is a viscosity solution of \eqref{eq:integP}-\eqref{eq:integPini} can be found in \cite{BiKaMo10} in the case $m=2$ and in \cite{StTeVa16} for $m>1$.
\end{proof}

Arguing in a similar way, we get a result for initial data of the form $\mu_0=M\delta_0$.

\begin{theorem}\label{thm:v}
Assume $m\in (1,+\infty)$,  $\sigma\in (0,1)$, $s=1-\sigma$,  $d=1$, 
$\mu_0 = M\delta_0$, and $u$ is the weak solution of
\eqref{eq:maineq}-\eqref{eq:maineqini} given by Theorem
\ref{thm:maineq}.  If $v$ is the function defined in \eqref{def-v}
with $v_0$ redefined as
\[
v_0(x)= \int_{-\infty}^x \dd\mu_0(y)= 
\begin{cases}
0  &\text{if} \quad x<0,\\
M & \text{if} \quad x\geq0,
\end{cases}
\]
then $v\in C_b(\overline{Q}_T\setminus(0,0))$ is the unique
discontinuous viscosity solution of 
\eqref{eq:integP}-\eqref{eq:integPini} and satisfies:
\medskip
\begin{enumerate}[(i)]
\item[(i)\ ] $0\leq v \leq M$ in $Q_T$,
  \medskip
\item[(ii)\,] $v(\cdot,t)$ is nondecreasing for $t\geq0$,
  \medskip
\item[(iii)]
\(|v(x_2,t)-v(x_1,t)|\leq C_{1,s,m} \, t^{-\gamma}
M^{\delta} |x_2-x_1|\),\qquad for\quad $t>0$\quad
and\quad $x_1,x_2\in\R$, 
\medskip
\item[] where $\gamma,\delta>0$ are defined in Theorem \ref{thm:maineq}.
\end{enumerate}
Moreover, the \eqref{eq:integP}-\eqref{eq:integPini} has a comparison principle for  sub and super solutions.

\end{theorem}

Finally we mention that initial data $v_0$ does not need to be monotone in order to have  well-posedness.
\begin{theorem}\label{thm:strongCP}
Assume $m\geq 1$, $s\in(0,1)$, and $v_0\in BUC(\R)$ (resp. $v_0\in W^{1,\infty}(\R)$). Then there exists a unique viscosity
solution $v\in BUC(\overline Q_T)$ (resp. $v\in C([0,T]; W^{1,\infty}(\R))$) of \eqref{eq:integP}-\eqref{eq:integPini}.
\end{theorem} 
Uniqueness follows from Theorem \ref{thm:strongCP} in Section
\ref{sec:theoresults}, existence of nondecreasing solutions follows
from Theorem \ref{thm:main} (a), and general existence can be obtained
using Perron's method. 

  \begin{proof}[Proof of Theorem \ref{thm:maineq}\eqref{thm:maineq-itemtight}]
  Let $v_0:=\int_{-\infty}^x \dd \mu_0(y)$ and $v(x,t)=\int_{-\infty}^x
  u(y,t)\dd y$ be the corresponding solution of
  \eqref{eq:integP}-\eqref{eq:integPini}. For $r>0$, we define
  \[
  v_{0,r}(x)=\left\{\begin{split}
  0 \quad \textup{if}& \quad x<r,\\
  v_0(r) \quad \textup{if}&\quad  x\geq r,
  \end{split}\right.
  \]
  and let $v_r$ be the corresponding solution of
  \eqref{eq:integP}-\eqref{eq:integPini}. Note that $v_r$ is the
  integral of the solution $u_r$ of
  \eqref{eq:maineq}-\eqref{eq:maineqini} with initial data $\mu_{0,r}=
  v_0(r)\delta_r $. By finite speed of propagation (Theorem \ref{thm:maineq}\eqref{thm:maineq-itemFSP}), there exists $R=R(r,T,u_0)>0$ such that
  $\supp \{u_r(\cdot,t)\}\subset (-\infty,R]$ for $t\in[0,T]$. By preservation of mass
  (Theorem \ref{thm:maineq}\eqref{thm:maineq-itemA}), this implies that
\begin{equation}\label{eq:fspv}
  v_r(x,t)=\int_{-\infty}^x u_r(y,t)\, dy=\|u_r(\cdot,t)\|_{L^1(\R)}= \mu_{0,r}(\R)= v_0(r) \qquad
  \textup{for} \qquad  x\geq R \quad \textup{and} \quad t\in[0,T].
\end{equation}
  Moreover, since $v_{0,r}\leq v_0$ then, by comparison (Theorem \ref{thm:strongCP} and Theorem \ref{thm:comppp2}) we have that $v_r\leq v$. We are ready to
  prove the result. Fix $\veps>0$ and choose $r=r(\veps, u_0)$ such
  that $\mu_0(\R)< v_0(r)+\veps/2$. Thus, 
  \[
  \mu_0(\R)< v_{0,r}(x)+\veps/2\quad \textup{for all} \quad x\geq r.
  \]
  By \eqref{eq:fspv} there exists $R>0$ such that for all $t\in [0,T]$ we have
  \[
   \mu_0(\R)< v_r(R,t) + \veps/2.
  \]
  Finally, using comparison between $v_r$ and $v$, and conservation of mass of $u$, we get
  \[
  \begin{split}
  \mu_0(\R)&< v(R,t)+ \veps/2=\int_{-\infty}^R u(x,t)\dd x+\veps/2=  \|u(\cdot,t)\|_{L^1(\R)} - \int_{R}^\infty u(x,t)\dd x +\veps/2\\
   &= \mu_0(\R) - \int_{R}^\infty u(x,t)\dd x +\veps/2,
  \end{split}
  \]
that is,
\[
\int_{R}^{\infty} u(x,t) \dd x\leq \veps/2.
\]
In a similar way, one gets $\int_{-\infty}^{R} u(x,t) \dd x\leq \veps/2$. Since the computations above are uniform in $t\in[0,T]$, the result follows.
  \end{proof}

\section{Numerical scheme and main results}\label{sec:main}

We propose now an explicit finite difference approximation for the
integrated equation \eqref{eq:integP}, based on first finding a good
discretization of the quasilinear fractional operator
\begin{equation}\label{eq:op}
L^s[\phi](x):= - |\partial_x \phi(x)|^{m-1} \flap \phi(x),
\end{equation}
and then discretizing in time and solving on a grid. 
Note that $L^s[\phi]$ is well defined for $\phi \in C^2_b(\R)$. 
\medskip

\subsubsection*{Discretization of $L^s$}

Let $h>0$.
(i) We discretize $\flap$ by a (symmetric) monotone discretization
\begin{equation}\label{eq:discretization}
\flap_h \phi(x)= \sum_{k\not=0} (\phi(x)-\phi(x+x_k))
\omega_{k}\qquad\text{where}\qquad
\omega_{k}=\omega_{-k}\geq0\quad\text{for}\quad  k\in \Z.
\end{equation}
This is a quadrature approximation where the points $x_k\in\R$ and
weights $\omega_k$ will depend on $h$. We will assume the following
consistency and monotonicity assumptions:

\begin{equation}
\label{as:cons} \tag{$A_{\textup{c}}$}  \hspace{-3cm}\textup{$(-\Delta)_h^{s}\phi \to (-\Delta)^{s}\phi$\quad as\quad $h\to0^+$\quad
 locally uniformly for $\phi\in C^2\cap C_b(\R)$.}
 \end{equation}
 \begin{equation}
\label{as:m} \tag{$A_{\textup{m}}$} \textup{$\exists\,  C_s>0$ such
  that}\,  \sum_{k\not=0} \omega_k \leq C_s h^{-2s}, \,  \sum_{ |kh|>1} \omega_k \leq C_s,
 \, \sum_{0<
  |kh|\leq1} |hk|\omega_k \leq C_s \left\{ \begin{split}&h^{1-2s} \hspace{6.24mm}   \textup{if} \, \,\,  s>\tfrac{1}{2}\\
 & |\log(h)| \, \, \, \, \, \textup{if} \, \,\,    s=\tfrac{1}{2}\\
 & 1  \hspace{13mm}  \textup{if} \, \,\,  s<\tfrac{1}{2}.
\end{split}\right.
\quad   \end{equation}
Such discretizations can be found e.g. in
\cite{dTEnJa18,dTEnJa19}, see Appendix \ref{app:disfl} for examples and comments.
\medskip

\noindent (ii) We discretize $\partial_x\phi$ by up-wind finite differences
\begin{equation}\label{eq:discDer}
 D_h\phi(x):= \left\{ \begin{array}{ll}
 \displaystyle\frac{\phi(x+h)-\phi(x)}{h}    &\text{if }\quad  (-\Delta)_h^{s}\phi(x) \leq 0, \\[2mm]
 \displaystyle  \frac{\phi(x)-\phi(x-h)}{h} &\text{if } \quad (-\Delta)_h^{s}\phi(x) > 0.
    \end{array}\right.
\end{equation}
This approximation is 
consistent, $D_h\phi(x)\to \partial_x\phi(x)$ locally uniformly as $h\to 0$ for $\phi\in C^1$. 
\medskip

\noindent (iii) The discretization of the quasilinear operator \eqref{eq:op} is then given by
\begin{equation}\label{eq:discop}
L_h^s[\phi](x)=- \max\{D_h\phi(x),0\}^{m-1} \flap_h \phi(x).
\end{equation}
 Below we will see that this approximation is monotone. 
Note that $\max\{D_h\phi(x),0\}^{m-1}=|D_h\phi(x)|^{m-1}$ when $\phi$ is nondecreasing. Under assumption \eqref{as:cons}, the  approximation is  therefore  conditionally  consistent in the sense that
\begin{align}
\textup{$L_h^{s}[\phi] \to L^{s}[\phi]$\quad as\quad $h\to0^+$\quad
 locally uniformly for {\em nondecreasing} $\phi\in C^2\cap C_b(\R)$.} 
  \end{align}
 In this paper solutions of \eqref{eq:integP} and its approximation will always be nondecreasing (they are cumulative density functions). 
\medskip

\subsubsection*{Grid and discretization of full integrated problem}
 Let $J\in\N$, $\tau=\frac TJ>0$, and define
equidistant grids
\begin{align}
  \mathcal{T}_{\tau}=\{t_j:=j\tau\}_{j=0}^{J} \qquad\text{and} \qquad h\Z =\{z_i:=h i: i\in \Z\}.
      \end{align}
On this grid we discretize
\eqref{eq:integP}-\eqref{eq:integPini} in space by $L_h^s$ and in time
by the forward Euler method,  where we interpret $L_h^s$ as an operator on grid functions $\phi:h\Z^d\to \R$, 
$L_h^s[\phi]_i=- \max\{D_h\phi_i,0\}^{m-1} \flap_h \phi_i$ for 
\begin{equation}\label{eq:defopdisc}
\flap_h \phi_i= \sum_{k\not=0} (\phi_i-\phi_{i+k})\omega_k\quad \text{and}\quad D_h\phi_i:= \left\{ \begin{array}{ll}
 \displaystyle\frac{\phi_{i+1}-\phi_i}{h}    &\text{if }\quad  (-\Delta)_h^{s}\phi_i \leq 0, \\[2mm]
 \displaystyle  \frac{\phi_i-\phi_{i-1}}{h} &\text{if } \quad (-\Delta)_h^{s}\phi_i > 0.
    \end{array}\right.
\end{equation}

The numerical approximation of \eqref{eq:integP}-\eqref{eq:integPini} is a function $V:h\Z \times \mathcal{T}_{\tau} \to \R$ given by
\begin{align}\label{eq:nums}\tag{S}
  V_i^{j+1}&=V_i^j + \tau L_h^s [V_\cdot^{j}]_i, && i\in \Z, \quad j\in \N,\\
  V_i^0 &=v_0(x_i),&& i\in\Z,\tag{IC$_S$}\label{eq:numic} 
\end{align}

To ensure monotonicity and stability of the numerical scheme, we need
CFL-conditions, and these conditions depend on the regularity of the initial
data:
\begin{align}\label{as:CFL1}\tag{CFL1}
&\textup{If $\|v_0\|_\infty\leq M$}, && \text{then\quad $\tau \leq C_1 h^{2s+m-1}$}&& \text{where\quad $C_1^{-1}=C_s m (2M)^{m-1}$}.\\
\label{as:CFL2}\tag{CFL2}
&\textup{If $\|v_0\|_\infty\leq M$, $\|\partial_x
v_0\|_\infty\leq L$}, && \text{then\quad $\tau \leq C_2 h^{2s\vee1}$} f_s(h)&& \text{where\quad
$C_2^{-1}=C_s m L^{m-2}  L\vee (2M)$}
\end{align}
with $f_s(h)=1$ if $s\not=\frac{1}{2}$ and $f_s(h)=\frac{1}{|\log(h)|}$ if $s=\frac{1}{2}$. 
These conditions give constraints on $\tau$ in terms of $h$, where
  \eqref{as:CFL1} is more restrictive than
  \eqref{as:CFL2}. The improvement in \eqref{as:CFL2} is due to
  better regularity of solutions. When $s< \frac{1}{2}$,  the highest
  order term in the equation is $v_x$, and \eqref{as:CFL2}
  imposes $\tau =O(h)$, the hyperbolic scaling. When $s>\frac{1}{2}$,
  \eqref{as:CFL2} is consistent with the optimal fractional parabolic
  scaling, $\tau = O(h^{2s})$. 
 We also note that in general the scheme \eqref{eq:nums} is consistent with the equation
\begin{align}\tag{IP'}\label{eq:integPp}v_t+ \max\{v_x,0\}^{m-1} (-\Delta)^{s}v=0.
\end{align}
However, for the nondecreasing solutions considered in this paper, this equation coincides with \eqref{eq:integP} (see Lemma \ref{lem:coincide} below).

\subsubsection*{Approximation of solutions of the original problem} From
the approximation $V$ of \eqref{eq:integP}-\eqref{eq:integPini} we get
an approximation $U$ of \eqref{eq:maineq}-\eqref{eq:maineqini} by numerical differentiation:
\begin{align}\label{defU}
U_i^j= \frac{V_i^j-V_{i-1}^j}{h} \qquad\text{for}\qquad i\in\Z,\ j=1,\dots,J,
\end{align}
or equivalently
 \begin{align}\label{defU2}
 V_i^j =h\sum_{k=-\infty}^iU_k^j \qquad\text{for}\qquad  j=1,\dots,J.
\end{align}
Note that $\|\mu_0\|_1\leq M$ and $\|\mu_0\|_\infty\leq L$ immediately
implies that $\|v_0\|_\infty\leq M$ and $\|\partial_x v_0\|_\infty\leq
L$.  See the end of this section for an interpretation of this
approximation scheme for $u$ as an conservative upwind finite volume scheme for \eqref{eq:maineq}.
\bigskip

We now present our main results, the convergence of the schemes. Other
properties of the scheme and a priori estimates will be given in 
Section \ref{sec:propNumSch} and 
Section \ref{sec:propUandconv}.
We need piecewise linear and piecewise constant space-time
interpolants $\overline{V},\overline{U}:\overline{Q}_T\to\R$ of approximate solutions $V$,
$U$ of \eqref{eq:nums}, \eqref{defU}:
\begin{align*}
\overline{V}_h(x,t) & = \frac{x_i-x}{h} V_{i-1}^j +
\frac{x-x_{i-1}}{h} V_i^j  &&\textup{for} \quad
(x,t)\in[x_{i-1},x_i)\times[t_j,t_{j+1}), \\[0.2cm]
\overline{U}_h(x,t)& =U_i^j &&\textup{for} \quad (x,t)\in[x_{i-1},x_i)\times[t_j,t_{j+1}).
\end{align*}
\begin{remark}\label{rem:intutov}
We immediatly have $\partial_x \overline{V}_h=\overline{U}$ a.e. and
in $\mathcal D'$. Moreover, in Section \ref{sec:propUandconv}, we show 
  \[
\overline{V}_h(x,t)=\int_{-\infty}^x \overline{U}_h(y,t) \dd
y\qquad\text{for}\qquad (x,t) \in \overline{Q}_T.
\]
\end{remark}

We give our main convergence results for nondecreasing solutions
of the scheme for the integrated problem
\eqref{eq:integP}-\eqref{eq:integPini}. There are three different results
depending on the regularity of the initial data: (i) $v_0\in BUC(\R)$,
(ii) $v_0\in W^{1,\infty}(\R)$, and (iii) $v_0$ a discontinuous
step-function. 

\begin{theorem}\label{thm:main}
Assume $s\in(0,1)$, $m\geq2$,  $v_0$ nondecreasing, \eqref{as:cons}, 
\eqref{as:m}, and $\{V_h\}_{h>0}$ is a family of solutions of the scheme
\eqref{eq:nums}-\eqref{eq:numic}.
\medskip

\noindent (a) (Continuous case) Assuming $v_0\in BUC(\R)$
 and \eqref{as:CFL1}
 (resp. $v_0\in W^{1,\infty}(\R)$ and \eqref{as:CFL2}), then:
 \smallskip
\begin{enumerate}
\item[(i)]\label{thm:main-item1} (Convergence) There exists a limit $v\in BUC(\overline Q_T)$
  ($v\in C([0,T];W^{1,\infty}(\R))$) such that
 $$\overline{V}_h\stackrel{h\to0}{\longrightarrow}v \qquad \textup{ locally uniformly\qquad in } \overline{Q_T}.$$
\item[(ii)]\label{thm:main-item1b} (Limit) The limit $v$
   from part (i) is the unique viscosity solution of \eqref{eq:integP}-\eqref{eq:integPini}. 
\end{enumerate}
\medskip

\noindent (b) (Jump discontinuity) We assume \eqref{as:CFL1} and for $M>0$
and $a\in \R$,
\[
v_0(x)= 
\begin{cases}
0 & \text{if} \quad x<a,\\
M & \text{if} \quad x\geq a.
\end{cases}
\]
\begin{enumerate}
\item[(i)]\label{thm:main-item1} (Convergence) There exists a limit
  $v\in C_b(\overline Q_T\setminus\{(a,0)\})$,  such that
 $$\overline{V}_h\stackrel{h\to0}{\longrightarrow}v \qquad \textup{ locally uniformly\qquad in } \overline{Q_T}\setminus\{(a,0)\}.$$
\item[(ii)]\label{thm:main-item1b} (Limit) The limit $v$
   from part (i) is the unique discontinuous viscosity solution of \eqref{eq:integP}-\eqref{eq:integPini}. 
\end{enumerate}
\end{theorem}
This result is proved using monotonicity arguments, viscosity solutions, and half-relaxed
limits, see Section \ref{sec:propNumSch} for the details and
slightly more general results.

\begin{remark}
In addition to proving convergence, even in the discontinuous case,
this result also proves the existence (and continuity) of
(discontinuous) nondecreasing viscosity solutions of
\eqref{eq:integP}-\eqref{eq:integPini}. However, in the discontinuous 
case, the proof uses the existence of
a fundamental solution of \eqref{eq:maineq}. 
  \end{remark}

Now we give the convergence results for the approximations of
nonnegative solutions of the fractional pressure equation
\eqref{eq:maineq}-\eqref{eq:maineqini}. The convergence obtained is
given in terms of the Rubinstein-Kantorovich distance defined for 
nonnegative functions $f_1,f_2\in L^1(\R)$  with the same total
mass as 
\[
d_0(f_1,f_2)=\sup_{\varphi \in \textup{Lip}_{1,1}(\R)} \int_{\R}( f_1(x)-f_2(x))\varphi(x) \dd x,
\]
where $\textup{Lip}_{1,1}(\R)=\{\varphi\in C(\R):
\|\varphi\|_{L^\infty}\leq 1,\ \|D\varphi\|_{L^\infty}\leq 1\}$ and
$D$ is the weak derivative.

There are three types of results,
depending on the regularity of the initial data $\mu_0$: (i) $\mu_0\in
L^1$, (ii) $\mu_0\in
L^1\cap L^\infty$, and (iii) $\mu_0$ a point mass and $u$ a fundamental solution.
\begin{theorem}\label{thm:main2}
Assume  $\sigma\in(0,1)$, $s=1-\sigma$,  $m\geq2$, $\mu_0$ nonnegative, \eqref{as:cons},
\eqref{as:m}, $\{V_h\}_{h>0}$ a family of solutions of the scheme
\eqref{eq:nums}-\eqref{eq:numic} with $v_0=\int_{-\infty}^x
\mu_0(y)\dd y$, $U_h$ is the numerical derivative of
$V_h$ in \eqref{defU}, and $u$ the weak solution of
\eqref{eq:maineq}-\eqref{eq:maineqini}. 
\smallskip

\begin{enumerate}[(a)]
\item\label{thm:main-item21} ($L^1$ case) Assume $\mu_0\in L^1(\R)$ and
\eqref{as:CFL1} (resp. $\mu_0\in
L^1(\R)\cap L^\infty(\R)$ and \eqref{as:CFL2}), then $0\leq u\in
L^\infty([0,T];L^1(\R))$ (resp. $0\leq u\in
  L^\infty([0,T];L^1\cap L^\infty(\R))$)   and we have
\begin{equation}\label{eq:weakconv}
\sup_{t\in[0,T]} d_0(\overline{U}_h(\cdot, t),u(\cdot,t)) \stackrel{h\to0}{\longrightarrow} 0.
\end{equation}  

\item\label{thm:main-item22} (Fundamental solution) Assume $\mu_0=M
  \delta_a$ and \eqref{as:CFL1}, then $0\leq u\in
  L^\infty((0,T];L^1(\R))$   and  
\begin{equation}\label{eq:weakconv2}
\sup_{t\in[0,T]}d_0(\overline{U}_h(\cdot, t),u(\cdot,t)) \stackrel{h\to0}{\longrightarrow} 0.
\end{equation}
\end{enumerate}
\end{theorem}

\begin{remark}\label{rem:linftest}
  (a) The $d_0$-distance metrizes weak ($*$) convergence in the space
of probability measures $P(\R)$. It differs (slightly) from the Wasserstein-1 distance
$d_1$ which also implies convergence of first-moments and requires possibly unbounded test function $\varphi$ with
$\|D\varphi\|_{L^\infty}\leq 1$.
  \medskip
  
\noindent (b) The convergence results in Theorem \ref{thm:main2} are equivalent (up to normalization) to
  convergence in the (complete metric) spaces $C([0,T]; P(\R))$. 
  \medskip
  
\noindent (c) When $\mu_0\in L^1(\R)\cap L^\infty(\R)$, we get better
convergence in space. A simple
density argument shows that we have 
convergence for test functions $\varphi\in L^\infty(\overline{Q_T})$: For a.e. $t\in[0,T]$,
\[
\left|\int_{\R} \left(\overline{U}_h(x,t)-u(x,t)\right)\varphi(x,t) \dd x\right| \stackrel{h\to0}{\longrightarrow} 0 \quad \textup{for all} \quad \varphi\in L^\infty(\overline{Q_T}).
\]
\end{remark}

\begin{remark}
(a) Theorem \ref{thm:main2} will follow as a consequence of Theorem
  \ref{thm:main}  and tightness arguments,  after taking
  $v_0(x)=\int_{-\infty}^x\dd\mu_0(y)$ and identifying $\overline U_h$
  and $u$ as derivatives of $\overline V_h$ and $v$.
 \medskip 
  
\noindent (b) With $v_0(x)=\int_{-\infty}^x\dd\mu_0(y)$, (a) - (b)
in Theorem \ref{thm:main2} are consequences of (a) - (b) in Theorem
\ref{thm:main}, see Section \ref{sec:propUandconv} for the detailed
proof.
\medskip

\noindent(c) As opposed to Theorem \ref{thm:main}, existence of
solutions does not follow, even if we have a limit and candidate
$u=\partial_x v$. We lack energy discrete estimates for approximate
solution to conclude that $\partial_x v$ is the weak solution of
\eqref{eq:maineq}-\eqref{eq:maineqini}.  
  \end{remark}

\subsection{An upwind finite volume scheme for \eqref{eq:maineq}}\label{sec:FVM}

Equation
\eqref{eq:maineq} is a conservation law of the form 
\begin{align}\label{CL}
u_t=- \nabla \cdot \mathfrak{F} \qquad \textup{for some flux} \quad \mathfrak{F}.
\end{align}
A (mass preserving) finite volume approximation of \eqref{CL} takes the following form in $d=1$,
\begin{align}\label{FVM}
U_i^{j+1}=U_i^j - \tau
\frac{F_{i+\frac{1}{2}}^j-F_{i-\frac{1}{2}}^j}{h}\qquad \textup{for
  some numerical flux} \quad F.
\end{align}
When $\mathfrak{F}= u \phi$ for
some velocity field $\phi$, the up-wind scheme \cite[Example
  1.1]{EGH00} follows from the choice
\[
F_{i+\frac{1}{2}}^j= U_{i}^j \big(\Phi_{i+\frac{1}{2}}^j\big)_+  + U_{i+1}^j \big(\Phi_{i+\frac{1}{2}}^j\big)_{-}, 
\]
where  $a_\pm=\pm\max\{\pm a,0\}$ and $\Phi$  is some numerical
approximation of $\phi$. The choice of $\Phi$
has a strong influence on the properties of the numerical method. We
refer to \cite{EGH00}  for details (see also \cite{CP04,BF12,CCH15,BCMS20}).

Our numerical scheme given by
\eqref{schemeVintro}-\eqref{schemeVintroini} and \eqref{SchemeUintro}
can be interpreted in a similar way. In our case
$\mathfrak{F}=u^{m-1}\phi$ for \(
\phi= - \nabla^{1-2\sigma} u= (-\Delta)^{s} v\) and $s=1-\sigma$, and
our discretization is of the form \eqref{FVM} with
\[
F_{i+\frac{1}{2}}^j= (U_{i}^j)^{m-1}
\big(\Phi_{i+\frac{1}{2}}^j\big)_+  + (U_{i+1}^j)^{m-1}
\big(\Phi_{i+\frac{1}{2}}^j\big)_{-}\qquad \text{and}\qquad
\Phi_{i+\frac{1}{2}}^j= (-\Delta_h)^s V_{i}^j.
\]
We can always write $\Phi$ in terms of $U$, and then we get a direct
approximation of $\phi$. Discretizing $(-\Delta)^s$ with \emph{powers of the discrete Laplacian} $(-\Delta_h)^s$ of (see \cite{Cia-etal18}
and Appendix \ref{app:disfl}), the resulting $\Phi$
becomes particularly simple
and suggestive: 
\[
(-\Delta_h)^s V_{i}^j = (-\Delta_h) (-\Delta_h)^{s-1} V_{i}^j= -D_h^+ D_{h}^- (-\Delta_h)^{-\sigma} V_{i}^j = -D_h^+  (-\Delta_h)^{-\sigma} D_{h}^- V_{i}^j  = -D_h^+  (-\Delta_h)^{-\sigma} U_i^{j},
\]
where $D^{\pm}_h$ are forward and backward finite difference approximations of $\nabla$ in dimension $d=1$.

\section{Viscosity solutions and comparison for the
  integrated problem
}\label{sec:theoresults}
Equation \eqref{eq:integP} is a degenerate quasi-linear PDE in
nondivergence form. It is known to be well-posed in the sense of
viscosity solutions. We refer to \cite{BiKaMo10} for the case $m=2$,
to \cite{StTeVa16} for $m\geq2$, and to \cite{ChJa17} for a more
general viscosity solution theory in multi-d that includes
\eqref{eq:integP} as a special case. In this section we define the
concepts of viscosity solutions and discontinuous viscosity
solutions. We also prove several properties that will be needed for
the convergence proof of our numerical scheme.  Because of later needs, we will state this theory for the (slightly) more general equation
\begin{align}\label{gIP}
 \tag{gIP} \partial_t v + F(\partial_x v) \flap v&=0\qquad \textup{in} \qquad Q_T, 
\end{align}
under the assumption 
\begin{align}\label{cond:F}
\tag{F}
0\leq F\in C(\R).    
\end{align}
\begin{remark}\label{rem:moregeneraldiffusioncoef}
We are interested in the two cases 
$F(\phi_x)=|\phi_x|^{m-1}$ (corresponding to equation \eqref{eq:integP}) and
$F(\phi_x)=\max\{\phi_x,0\}^{m-1}$ (corresponding to \eqref{eq:integPp} in Section \ref{sec:propNumSch}), which both trivially satisfy \eqref{cond:F} for $m\geq 1$. 
\end{remark}

We recall the concept of upper (lower) semicontinuous envelopes of locally bounded functions $f$:
\[
f^*(\xi)=\limsup_{\eta \to \xi} f(\eta) \quad \textup{and} \quad f_*(\xi)=\liminf_{\eta \to \xi} f(\eta).
\]
If $f$ is upper (lower) semicontinuous, then $f^*=f$ ($f_*=f$).

\begin{definition}\label{def:viscsol1}
 Assume  $s\in(0,1)$, \eqref{cond:F},  
and $v_0:\R\to\R$ be a bounded function. 

\begin{enumerate}[\noindent\rm (a)]
\item A bounded \textup{USC} (resp. \textup{LSC}) function $v:\overline{Q_T}\to \R$ is a \emph{viscosity subsolution} (resp. \emph{supersolution}) of 
\eqref{gIP}-\eqref{eq:integPini} if: 
\begin{enumerate}[\rm (i)]

\item  For any $\phi \in C^2\cap C_b (Q_T)$ and global maximum
 (resp. minimum) $(x_0,t_0)$ in $Q_T$ of $v-\phi$, 
\begin{equation}\label{eq:viscInt}
\phi_t(x_0,t_0)+ 
 F(\phi_x(x_0,t_0))(-\Delta)^{s}\phi (x_0,t_0) \leq 0 \quad \text{(resp. $\geq0$)}.
\end{equation}

\item For any $\phi \in C^2 \cap C_b (\overline{Q_T})$ and global
maximum (resp. minimum) $(x_0,0)$ in $\overline{Q_T}$ of $v-\phi$,
\begin{equation}\label{eq:viscIniCond}
\begin{split}
&\min\big\{\ v(x_0,0)-(v_0)^*(x),\ \phi_t(x_0,0)+ 
 F(\phi_x(x_0,0))
(-\Delta)^{s}\phi (x_0,0)\ \big\} \leq 0. \\
 \quad (\text{resp. }  &\max\big\{\ v(x_0,0)-(v_0)_*(x),\ \phi_t(x_0,0)+ 
  F (\phi_x(x_0,0))
(-\Delta)^{s}\phi (x_0,0)\ \big\} \geq 0.)
  \end{split}
\end{equation}
\end{enumerate}

\item A function $v\in  C_b (\overline Q_T)$ is a \emph{viscosity solution} if it is both a viscosity sub- and
supersolution.
\medskip

\item A bounded function $v$ is a \emph{discontinuous viscosity solution} if $v^*$ is a
viscosity subsolution and $v_*$ is a viscosity supersolution.
\end{enumerate}
\end{definition}

\begin{remark}\label{rem:testfunc}
 Without loss of generality, we can assume that $v(x_0,t_0)=\phi(x_0,t_0)$ in the previous definition. It then follows that  $v\leq\phi$ (resp. $v\geq\phi$). In fact we can take $\phi$ such that $v(x,t)<\phi(x,t)$  (resp. $v(x,t)>\phi(x,t)$) for $(x,t)\not=(x_0,t_0)$  and $v-\phi <1$  (resp. $v-\phi >1$) outside a neighborhood of  $(x_0,t_0)$.
\end{remark}

As the following result asserts, viscosity solutions in the sense of
Definition \ref{def:viscsol1} take initial conditions in the usual
sense. This is a simple (nonlocal) adaptation of a result in
\cite{Bar94}.

\begin{lemma}\label{lem:inidialData}
Assume  $s\in(0,1)$,  \eqref{cond:F},  
$v_0$ bounded, and $v$ is a bounded viscosity
subsolution (resp. supersolution) of
\eqref{gIP}-\eqref{eq:integPini}. Then 
\begin{equation}
v(x,0) \leq (v_0)^*(x) \quad \textup{(resp. $v(x,0)\geq (v_0)_*(x))$}.
\end{equation}
\end{lemma}

\begin{proof}
  We only give the proof 
for subsolutions since the argument for supersolutions is similar. 
Consider
$$\chi_{x,\epsilon, C}(y,t):=v(y,t)-\phi_{x,\epsilon, C}(y,t),$$
where the test function 
\[
\phi_{x,\epsilon, C}(y,t)= \frac{1}{\epsilon}\psi\left(y-x\right)+Ct,
\]
and $0\leq \psi \in C^2_b(\R)$ is radially nondecreasing
with $\psi(z)=|z|^2$ for $|z|\leq1$ and $\psi\equiv 2$ for $|z|\geq2$.
\smallskip

Take
$0<\epsilon<1/(2\|v\|_{L^\infty(\overline{Q_T})})$. 
For $|x-y|>1$, $\psi(x-y)>1$, and then by the bound on $\epsilon$,
\[ 
v(y,t)-\phi_{x,\epsilon, C}(y,t)\leq v(y,t)-
\frac{1}{\epsilon}\psi\left(y-x\right)<
-\|v\|_{L^\infty(\overline{Q_T})}.
\]
Since \(v(x,0)-\phi_{x,\epsilon, C}(x,0)=v(x,0)\geq-\|v\|_{L^\infty(\overline{Q_T})},
\) we find that $\chi_{x,\epsilon,
  C}(y,t)$ has a  global  maximum in
$\overline{Q_T}$ at a point $(y_0,t_0)$ with $|y_0-x|<1$. Moreover,
\[
v(y_0,t_0)-\phi_{x,\epsilon, C}(y_0,t_0)\geq v(x,0)-\phi_{x,\epsilon, C}(x,0)\geq -\|v\|_{L^\infty(\overline{Q_T})},
\]
so  $\phi_{x,\epsilon, C}(y_0,t_0) \leq
2\|v\|_{L^\infty(\overline{Q_T})}$, and by the definition of $\psi$,
\[
\frac{1}{\epsilon}|y_0-x|^2
+Ct_0=\phi_{x,\epsilon, C}(y_0,t_0) \leq
2\|v\|_{L^\infty(\overline{Q_T})}\qquad\implies\qquad |y_0-x|\leq \big(2\epsilon \|v\|_{L^\infty(\overline{Q_T})}\big)^{\frac12} .
\]
Now we let $\phi=\phi_{x,\epsilon, C}$ and take $C$ large such that
\[
C>L_\epsilon:=\|
 F(\phi_x)(-\Delta)^{s}\phi\|_{L^\infty(\overline
  Q_T)}
\]
where $L_\epsilon<\infty$ since $\phi\in
 C^2_b  (\overline{Q_T})$  and $F\in C(\R)$,  and
it does not depend on $C$ by spatial
differentiation.
Then $t_0=0$, since if $t_0>0$, we get a contradiction from the
subsolution inequality \eqref{eq:viscInt} for $v$:
\begin{equation*}
\begin{split}
0\geq\phi_t(y_0,t_0)+ 
 F(\phi_x(y_0,t_0))(-\Delta)^{s}\phi (y_0,t_0)\geq C - L_\epsilon>0.
\end{split}
\end{equation*} 
Since $t_0=0$ and 
\[
\phi_t(y_0,0)+ 
 F(\phi_x(y_0,0))(-\Delta)^{s}\phi
(y_0,0)\geq C-L_\epsilon>0,
\]
the initial condidition inequality \eqref{eq:viscIniCond} then implies that 
\begin{equation}\label{eq:bcarch}
v(y_0,0)\leq (v_0)^*(y_0).
\end{equation}
Moreover, $(y_0,0)$ is a global maximum of $\chi_{x,\epsilon,C}$, so
\begin{equation}\label{eqn}
\begin{split}
\chi_{x,\epsilon,C}(x,0)\leq  \chi_{x,\epsilon,C}(y_0,0)= v(y_0,0)-\frac{1}{\epsilon}\psi \left(y_0-x \right)\leq  v(y_0,0),
\end{split}
\end{equation}
and then since $y_0\to x$ as $\epsilon\to0$ and $v_0$ is USC,  we can take $\limsup$ in \eqref{eqn} and use \eqref{eq:bcarch} to see that 
\[
v(x,0)= 
 \chi_{x,\epsilon,C}(x,0)  \leq \limsup_{\epsilon\to0} (v_0)^*(y_0)\leq \limsup_{y\to x} (v_0)^*(y) \leq (v_0)^*(x).
\]
The proof is complete.
\end{proof}

As a consequence of the previous result and Theorem 6.1 in \cite{ChJa17}, we have the following strong comparison principle.

\begin{theorem}\label{thm:strongCP}
Assume $s\in(0,1)$, 
 \eqref{cond:F},  $v_0\in BUC(\R)$, and $v$ and $w$ are viscosity
sub and supersolutions of
\eqref{gIP}-\eqref{eq:integPini} respectively. Then  
\[
v\leq w \quad \textup{in} \quad \overline{Q_T}.
\]
\end{theorem}
\begin{proof}
Since $v_0$ is continuous, Lemma \ref{lem:inidialData} implies that $v(x,0)\leq v_0(x) \leq w(x,0)$.
This fact, together with the uniform continuity assumption on the initial data allow us to apply the result given by Theorem 6.1 in \cite{ChJa17} to get the $v\leq w$ in $\overline{Q_T}$.  See Section 3.2 in \cite{ChJa17} for how to rewrite  \eqref{gIP}  
into the form of \cite{ChJa17}. 
\end{proof}

Since we are also dealing with discontinuous viscosity solutions, we
need the following (partial) comparison result. This is a
generalization of Lemma 5.1 in \cite{BiKaMo10}, where the case $m=2$
is treated for a special type of discontinuous viscosity solution.

\begin{theorem}\label{thm:comppp2}
  Assume $s\in(0,1)$, 
   \eqref{cond:F},  $\mathcal{V}_0:\R\to \R$ bounded and
  nondecreasing, 
  and $v$, $w$, $\mathcal{V}$ are a viscosity
  subsolution, a viscosity supersolution, and a discontinuous
  viscosity solution of 
  \eqref{gIP}-\eqref{eq:integPini} with
  $v_0=\mathcal{V}_0$ in \eqref{eq:integPini}. Then 
\[ 
v\leq (\mathcal{V}_*)^* \quad \text{in}\quad Q_T\qquad \textup{and} \qquad
(\mathcal{V}^*)_*\leq w  \quad \text{in}\quad Q_T.
\]
Moreover, let $\mathcal{W}$ be another discontinuous viscosity
solution of 
\eqref{gIP}-\eqref{eq:integPini}. If $\mathcal{V}$ is continuous at a point $(x,t)$,
then $\mathcal{W}$ is continuous at $(x,t)$ and $\mathcal{W}(x,t)=\mathcal{V}(x,t)$. 
\end{theorem}
\begin{proof}
For $a>0$, define $w^a(x,t):=w(x+a,t)$ and note it is a viscosity
supersolution of 
\eqref{gIP}-\eqref{eq:integPini} with initial data
$\mathcal{V}_0(x+a)$ since the equation is translation
invariant. Then since $\mathcal{V}_0(x)\leq \mathcal{V}_0(x+a)$ and
$\mathcal{V}$ is a discontinuous viscosity solution, we can use Lemma
\ref{lem:inidialData} twice to conclude that 
\[
\mathcal{V}^*(x,0)\leq \mathcal{V}_0^*(x)\leq  (\mathcal{V}_0)_*(x+a)\leq w^a(x,0).
\]
Hence by the  properties  of $\mathcal{V}_0$, for any $a>0$ there is $v_{0,a}\in BUC(\R)$\footnote{ E.g. let $v_{0,a}=(\mathcal{V}_0^**\rho_{\frac a3})(x+\frac a3)$ for $0\leq \rho_{\frac a3}\in W^{1,1}(\R)$ with integral $1$ and support in $B(0,\frac a3)$. Since $\mathcal{V}_0$ is nondecreasing, direct computations show that $\|Dv_{0,a}\|_\infty\leq \|D\rho_{\frac a3}\|_1 \|\mathcal{V}_0^*\|_\infty<\infty$ and $\mathcal{V}_0(x)\leq\mathcal{V}_0^*(x)\leq v_{0,a}(x)\leq \mathcal{V}_0^*(x+\frac {2a} 3)\leq \mathcal{V}_0(x+a)$.} such that
\[
\mathcal{V}(x,0) \leq v_{0,a}(x) \leq w^a(x,0).
\]
We conclude first that $\mathcal{V}^*$ and $w^a$ are viscosity sub and
supersolutions to 
\eqref{gIP}-\eqref{eq:integPini} with initial condition
$v_{0,a}$, and then by comparison Theorem \ref{thm:strongCP}, that
$\mathcal{V}^*\leq w^a$ in $\overline{Q_T}$.
Then since $w$ is LSC, for any $(x,t)\in Q_T$,
$$(\mathcal{V}^*)_*(x,t)=\liminf_{(a,y,s)\to (0,x,t)}\mathcal{V}^*(y,s)\leq
\liminf_{(a,y,s)\to (0,x,t)}w^a(y,s)\leq\liminf_{(y,s)\to (x,t)}w(y,s)= w_*(x,t)=w(x,t).$$ 
In a similar way we can show that $v \leq (\mathcal{V}_*)^*$.

For the last part of the result, note that
$(\mathcal{V}^*)_*=(\mathcal{V}_*)^*=\mathcal{V}$ at $(x,t)$ by continuity of
$\mathcal V$. Then by the 
first part of this result, $\mathcal{W}^*\leq \mathcal{V}\leq
\mathcal{W}_*$ and hence $\mathcal{W}^*=\mathcal{W}_*=\mathcal{V}$ at $(x,t)$.
The proof is complete.
\end{proof}

By comparison, solutions with nondecreasing initial data will be nondecreasing.
\begin{lemma}\label{lem:incr}
Assume $s\in(0,1)$, 
   \eqref{cond:F}, $v_0$ is bounded, and $v$ is a bounded discontinuous viscosity solution of 
\eqref{gIP}-\eqref{eq:integPini}. If $v_0$ is nondecreasing, then $v(\cdot,t)$ is nondecreasing in the sense that 
$$v^*(x,t)\leq  (v_*)^*(x+\rho,t)\leq v^*(x+\rho,t)\qquad \text{and}\qquad v_*(x,t)\leq (v^*)_*(x,t)\leq  v_*(x+\rho,t),$$
for $\rho>0$, $x\in\R$, and $t>0$.
In particular, if $v$ is continuous, then 
$$v(x,t)\leq  v(x+\rho,t)\qquad \text{for}\qquad \rho\geq0, \ x\in \R, \ t>0.$$
\end{lemma}
\begin{proof}
    The second and third inequalities follow since $v_*\leq v\leq v^*$. The first  inequality follows from the comparison principle (Theorem 
    \ref{thm:comppp2}) since $w(x,t)=v^*(x+\rho,t)$ is a viscosity supersolution of 
    \eqref{gIP}-\eqref{eq:integPini}. The fourth inequality follows similarly.
\end{proof}

We now prove that nondecreasing solutions of equation \eqref{eq:integPp} also solve equation \eqref{eq:integP}.
\begin{lemma}\label{lem:coincide}
 Assume $s\in(0,1)$, $m\geq1$, and $v$ is a bounded nondecreasing function.
 Then $v$ is a discontinuous viscosity subsolution (resp. supersolution) of \eqref{eq:integPp}-\eqref{eq:integPini} if and only if it is a discontinuous viscosity subsolution (resp. supersolution)  of \eqref{eq:integP}-\eqref{eq:integPini}.
\end{lemma}
\begin{proof}
We only check the subsolution case since the supersolution one is similar. If $v^*-\phi$ has a max at $(x,t)$, then for every $\rho>0$ we have $(v^*-\phi)(x,t)\geq(v^*-\phi)(x+\rho,t)$ and hence
\begin{align*}
    \frac{\phi(x+\rho,t)-\phi(x,t)}{
    \rho} 
    \geq \frac{v^*(x+\rho,t)-v^*(x,t)}{
    \rho}\geq 0.
\end{align*}
Sending $\rho\to0$ shows that $\phi_x(x,t)\geq 0$. So at the maximum point, $\max\{\phi_x(x,t),0\}=\phi_x(x,t)=|\phi_x(x,t)|$, and then the viscosity subsolution inequalities for \eqref{eq:integP} and \eqref{eq:integPp} coincide.
\end{proof} 
 
 \section{Analysis of the scheme 
    for the integrated problem} 
 \label{sec:propNumSch} 

 In this section we show monotonicity, a priori estimates, boundary
 conditions at infinity, and consistency of the scheme
 \eqref{eq:integP}-\eqref{eq:integPini}. Then we prove local uniform
 convergence using the method of  half-relaxed limits and the strong
 comparison result of Section \ref{sec:theoresults}. We conclude the
 section with a proof of Theorem \ref{thm:main}.
 
 Following standard notation, 
 for a function $\phi: h\Z \to \R$, we denote 
\begin{equation}\label{eq:defLcal}
\LL(h,\phi_i,  \{\phi_k\}_{k\not=i} ):= L_h^s[\phi]_i,
\end{equation}
 where $L_h^s$ is defined by
\eqref{eq:defopdisc}. 
 For the scheme \eqref{eq:nums} we then use the following notation:
\[
V_i^{j+1}=\Se[V^j_{\cdot}]_i.
\]
where the one step propagator $\Se$ is defined as
\begin{align}\label{def-S}
\Se[\phi]_i:= \phi_i-\tau\max\{D_h\phi_i,0\}^{m-1} (-\Delta)_h^{s}\phi_i= \phi_i  +  \tau\LL(h,\phi_i,\{\phi_k\}_{k\not=i}).
\end{align}
 
 We first give a monotonicity and translation invariance result for $\LL$. 
 \begin{lemma}\label{lem:mono1}
   Assume $s\in(0,1)$, $m\geq 1$, 
   $h>0$, and $\phi,\psi\in \ell^\infty( h\Z)$ such that  $\phi_i\leq \psi_i $ for $i\in\Z$. 
Then we have the following properties for $\LL$:
\begin{enumerate}[(i)]
\item[(i)\ ]
  If $\phi_i=\psi_i=r$, 
  \quad then\quad
  $\LL(h,r, \{\phi_k\}_{k\not=i} )\leq \LL(h,r, \{\psi_k\}_{k\not=i} )$.
  \smallskip
\item[(ii)\,]
  If $ \{\phi_k\}_{k\not=i} =\{\psi_k\}_{k\not=i} =\{p_k\}_{k\not=i}$,
  \quad then\quad $\LL(h,\phi_i,\{p_k\}_{k\not=i})\geq
  \LL(h,\psi_i,\{p_k\}_{k\not=i})$.
    \smallskip
  \item[(iii)]
    If $c\in \R$,\quad then\quad $\LL(h,\phi_i+c, \{\phi_k+c\}_{k\not=i} )=\LL(h,\phi_i, \{\phi_k\}_{k\not=i} )$.
\end{enumerate}
 \end{lemma}
 
  \begin{proof}
 We only prove (a), since (b)
 follows similarly, and (c) is just a consequence
 of the definitions of $D_h$ and $\flap_h$.  Note that  since $\phi\leq \psi$  
 and $\phi_i=\psi_i=r$, we have 
 \begin{equation}\label{eq:compFL}
   -(-\Delta)_h^{s}\phi_i=\sum_{k\not=0} (\phi_{i+k}-r) \omega_k
   \leq\sum_{k\not=0}  (\psi_{i+k}-r) \omega_k=-(-\Delta)_h^{s}\psi_i.
\end{equation}
\emph{Case 1: $-\flap_h\phi_i\geq0$.} By \eqref{eq:compFL} we have
that $-(-\Delta)_h^{s}\psi_i\geq0$, and then  since $\phi\leq \psi$  
and $\phi_i=\psi_i=r$,
\begin{align*}
  \LL(h,r, \{\phi_k\}_{k\not=i} )&=-\max\left\{\frac{\phi_{i+1}-r}{h},{0}  \right\}^{m-1}(-\Delta)_h^{s}\phi_i\\
  &\leq -\max\left\{\frac{\psi_{i+1}-r}{h},{0} \right\}^{m-1}(-\Delta)_h^{s}\psi_i=\LL(h,r, \{\psi_k\}_{k\not=i} ).
\end{align*}
\emph{Case 2: $-\flap_h\phi_i<0$.} If $-(-\Delta)_h^{s}\psi_i\geq0$ then, $
\LL(h,r, \{\phi_k\}_{k\not=i} )<0\leq \LL(h,r, \{\psi_k\}_{k\not=i} )$ and we are done. On the other
hand, if $-(-\Delta)_h^{s}\psi_i<0$, we use \eqref{eq:compFL},  $\phi\leq \psi$, 
and  $\phi_i=\psi_i=r$ to get
\begin{align*}
  \LL(h,r, \{\phi_k\}_{k\not=i} )&=-\max\left\{\frac{r-\phi_{i-1}}{h},{0} \right\}^{m-1}(-\Delta)_h^{s}\phi_i\\
  &\leq -\max\left\{\frac{r-\psi_{i-1}}{h},{0} \right\}^{m-1}(-\Delta)_h^{s}\psi_i=\LL(h,r, \{\psi_k\}_{k\not=i} ).\qedhere
\end{align*}
\end{proof}

We now prove a comparison result for the scheme
\eqref{eq:nums}-\eqref{eq:numic} via its one-step propagator $\Se$ in \eqref{def-S}. 
\begin{theorem}[Comparison]\label{theo:mono}
Assume $s\in(0,1)$, $m\geq2$, $h,\tau>0$, \eqref{as:m}, and
   $\phi,\psi:h\Z \to \R$  satisfy $\phi_i\leq \psi_i$ for all $i\in \Z$. 
If one of the following assumptions holds:
\begin{align}\label{lem:mono-item4-bis}\tag{CFLi}
&\textup{$|\phi_i|,|\psi_i|\leq M$}, \ i\in\Z, &&\hspace{-1cm}\text{and\qquad $\tau
    \leq C_1 h^{2s+m-1}$},\\
\label{lem:mono-item5-bis}\tag{CFLii}
&\textup{$|\phi_i|,|\psi_i|\leq M$,\
  $\frac{\phi_{i+1}-\phi_{i}}{h},\frac{\psi_{i+1}-\psi_{i}}{h}\leq
  L$},\ i\in \Z, && \hspace{-1cm}\text{and\qquad
  $\tau \leq C_2 h^{2s\vee1}f_s(h)$},
\end{align}
with $f_s(h)=1$ if $s\not=\frac{1}{2}$ and $f_s(h)=\frac{1}{|\log(h)|}$ if $s=\frac{1}{2}$,  $C_1^{-1}=C_s m (2M)^{m-1}$ and $C_2^{-1}=C_s m L^{m-2}  (L\vee 2
M)$, then
\begin{equation}\label{eq:monos}
\Se[\phi]_i \leq \Se[\psi]_i \quad \textup{for all} \quad i\in \Z.
\end{equation}
\end{theorem}
\begin{proof}
By definition of $\Se$,
\begin{align}\label{eq:diferenceSem}
&\Se[\phi]_i-\Se[\psi]_i\\
&= (\phi_i-\psi_i) + \tau \max\{D_h\phi_i,0\}^{m-1} \sum_{k\not=0}(\phi_{i+k}-\phi_{i})\omega_k - 
\tau \max\{D_h\psi_i,0\}^{m-1} \sum_{k\not=0}(\psi_{i+k}-\psi_{i})\omega_k .\nonumber
\end{align}
The following consequence of $\phi\leq\psi$  
will be useful: 
\begin{equation}\label{eq:estimdif2}
\displaystyle\sum_{k\not=0}(\phi_{i+k}-\phi_{i})\omega_k-\sum_{k\not=0}(\psi_{i+k}-\psi_{i})\omega_k \leq (\psi_i-\phi_i)\sum_{k\not=0}\omega_k .
\end{equation}
We will also use the notation
\[
\mathcal{R}:= \sup_{i\in \Z} \{ \max\{D_h\psi_i, D_h\phi_i, 0 \}\}.
\]
Note that $\mathcal{R}\geq0$, 
$\mathcal{R}\leq
2M/h$ if \eqref{lem:mono-item4-bis} holds, and $\mathcal{R}\leq L$ if
\eqref{lem:mono-item5-bis} holds. 
Now consider different cases:
\medskip

\noindent\emph{Case 1: $\sum_{k\not=0}(\phi_{i+k}-\phi_{i})\omega_k\leq0\leq
  \sum_{k\not=0}(\psi_{i+k}-\psi_{i})\omega_k$}. From
\eqref{eq:diferenceSem} and  $\phi\leq\psi$  we then get
\[
\Se[\phi]_i- \Se[\psi]_i\leq  (\phi_i-\psi_i)+0+0 \leq0.
\]
\smallskip

\noindent\emph{Case 2:
  $\sum_{k\not=0}(\phi_{i+k}-\phi_{i})\omega_k\geq0\geq
  \sum_{k\not=0}(\psi_{i+k}-\psi_{i})\omega_k$}.  From
\eqref{eq:diferenceSem}, \eqref{eq:estimdif2}, and  $\phi\leq\psi$   we then have
\begin{equation*}
\begin{split}
\Se[\phi]_i-\Se[\psi]_i&\leq(\phi_i-\psi_i)  +\tau  \mathcal{R}^{m-1}\bigg( \sum_{k\not=0}(\phi_{i+k}-\phi_i)\omega_k - \sum_{k\not=0}(\psi_{i+k}-\psi_i)\omega_k  \bigg)\\
&\leq (\phi_i-\psi_i)\bigg(1- \tau \mathcal{R}^{m-1} \sum_{k\not=0} \omega_k \bigg)=:c(\phi_i-\psi_i).
\end{split}
\end{equation*}
By \eqref{as:m} and \eqref{lem:mono-item4-bis},
\(
c \geq  1 - (2M)^{m-1} C_s  \frac{\tau}{h^{2s+m-1}}\geq 0,
\)
and by \eqref{as:m} and \eqref{lem:mono-item5-bis},
\(
c \geq 1  - L^{m-1} C_s  \frac{\tau}{h^{2s}}\geq 0.
\)
In both cases $\Se[\phi]_i-\Se[\psi]_i\leq 0$ by (ii).
\medskip

\noindent\emph{Case 3: $\sum_{k\not=0}(\phi_{i+k}-\phi_{i})\omega_k\geq0$ and $ \sum_{k\not=0}(\psi_{i+k}-\psi_{i})\omega_k\geq0$}. In this case, 
\[
D_h\phi_i=(\phi_{i+1}-\phi_i)/h \quad \textup{and} \quad D_h\psi_i=(\psi_{i+1}-\psi_i)/h.
\]
From \eqref{eq:diferenceSem} we have
\[
\Se[\phi]_i-\Se[\psi]_i=(\phi_i-\psi_i)+ M_1+ M_2,
\]
with 
\begin{align*}
M_1&= \tau
\max\{D_h\phi_i,0\}^{m-1}\bigg(\sum_{k\not=0}(\phi_{i+k}-\phi_{i})\omega_k-
\sum_{k\not=0}(\psi_{i+k}-\psi_{i})\omega_k\bigg),\\
M_2&=\tau\Big(\max\{D_h\phi_i,0\}^{m-1}- \max\{D_h\psi_i,0\}^{m-1}\Big) \sum_{k\not=0}(\psi_{i+k}-\psi_{i})\omega_k.
\end{align*}
From \eqref{eq:estimdif2} and  $\phi\leq\psi$  we estimate (as in Case 2)
\begin{align*}
M_1\leq &\ \tau \max\{D_h\phi_i,0\}^{m-1} (\psi_i-\phi_i)
\sum_{k\not=0}\omega_k \leq
(\psi_i-\phi_i) \tau \mathcal{R}^{m-1} \sum_{k\not=0} \omega_k \\
\leq &\
-(\phi_i-\psi_i) \begin{cases}  (2M)^{m-1} C_s  \frac{\tau}{h^{2s+m-1}}
  &\text{if \eqref{lem:mono-item4-bis} holds,}\\[0.1cm]
  L^{m-1} C_s  \frac{\tau}{h^{2s}} &\text{if \eqref{lem:mono-item5-bis} holds.}
  \end{cases}
\end{align*}
Now we estimate $M_2$. 
If $\max\{D_h\phi_i,0\}\leq \max\{D_h\psi_i,0\}$ then $M_2\leq0$. If 
$\max\{D_h\phi_i,0\}>\max\{D_h\psi_i,0\}$, we can use 
the convex inequality
$|a|^p - |b|^p \leq p a^{p-1} |a-b|$ for $p\geq1$ and  $\phi\leq\psi$  to find that

\begin{align*}
\max\{D_h\phi_i,0\}^{m-1}- \max\{D_h\psi_i,0\}^{m-1} 
&\leq (m-1) \mathcal{R}^{m-2} \max\Big\{\frac{(\phi_{i+1}-\phi_i) - (\psi_{i+1}-\psi_i)}h,0\Big\}\\
&\leq - (\phi_i-\psi_i) (m-1) \frac{\mathcal{R}^{m-2}}{h}
\end{align*}

Since $p=m-1$, we must assume $m\geq2$ for the inequality to hold. If
\eqref{lem:mono-item4-bis} holds, we find that 
\begin{align*}
&\mathcal{Q} :=\sum_{k\not=0}(\psi_{i+k}-\psi_{i})\omega_k
\leq
2M\sum_{k\not=0} \omega_k \leq 2M C_s h^{-2s}, \quad \text{and then}
\\
&M_2 \leq - (\phi_i-\psi_i) (m-1) \tau \frac{\mathcal{R}^{m-2}}{h} \mathcal{Q} \leq - (\phi_i-\psi_i) (m-1) (2M)^{m-1}C_s\frac{\tau}{h^{2s+m-1}}.
\end{align*}
If \eqref{lem:mono-item5-bis} holds and $s\not=1/2$, 
\begin{align*}
&\mathcal{Q}  \leq L \sum_{0<|kh|\leq1}|x_k| \omega_k +2M \sum_{|kh|>1}
\omega_k \leq C_s\big(L\, (h^{1-2s}\vee 1)+ 2M\big),\\[0.1cm]
\begin{split}
&M_2 \leq  - (\phi_i-\psi_i) (m-1) L^{m-2} \frac{\tau}{h} C_s\big(L\,(h^{1-2s}\vee1)+ 2M\big)\\
&\quad\ \,\leq- (\phi_i-\psi_i) (m-1) L^{m-2}  (L\vee 2M) C_s\tau
\,\Big(\frac{1}{h}\vee \frac{1}{h^{2s}}\Big).
\end{split}
\end{align*}
Combining all estimates of Case 3, and using that $\frac{1}{h}\vee
\frac{1}{h^{2s}}=\frac1{h^{\max\{1,2s\}}}$ for $h\in(0,1)$, we have
\[
\Se[\phi]_i-\Se[\psi]_i \leq (\phi_i-\psi_i)\begin{cases}(1- K_1
  \frac{\tau}{h^{2s+m-1}}),& \text{if \eqref{lem:mono-item4-bis} holds}, \\[0.1cm]
  (1- K_2 \frac{\tau}{h^{\max\{1,2s\}}}),& \text{if \eqref{lem:mono-item5-bis} holds and $s\not=\frac{1}{2}$},
  \end{cases}
\]
with $K_1=C_s m (2M)^{m-1}$ and $K_2=C_s m L^{m-2}  (L\vee2M)$. The
right hand side is negative since  $\phi\leq\psi$  and the relation between $\tau$
and $h$ given by \eqref{lem:mono-item4-bis} or \eqref{lem:mono-item5-bis}. The case \eqref{lem:mono-item5-bis} with $s=\frac{1}{2}$ follows in a similar way. 
\medskip

\noindent\emph{Case 4:
  $\sum_{k\not=0}(\phi_{i+k}-\phi_{i})\omega_k\leq0$ and $
  \sum_{k\not=0}(\psi_{i+k}-\psi_{i})\omega_k\leq0$}. This case is
similar to the Case 3 and we omit the proof.
\end{proof}

From comparison and translation properties of $\Se$ we now
prove the following preservation properties and a priori estimates 
for the scheme. 
\begin{corollary}\label{coro:prese}
Assume $s\in(0,1)$, $m\geq2$, $h,\tau>0$, \eqref{as:m}, and
$\{V^0_i\}_{i\in \mathbb{Z}}$ and $\{W^0_i\}_{i\in \mathbb{Z}}$
are nondecreasing and satisfy one of the following assumptions:
\smallskip
\begin{align}\label{lem:mono-item4}\tag{CFLa}
&\textup{$|V^0_i|,|W^0_i|\leq M$}, \ i\in\Z, &&\hspace{-1cm}\text{and\qquad $\tau
    \leq C_1 h^{2s+m-1}$},\\[0.1cm]
  \label{lem:mono-item5}\tag{CFLb}
&\textup{$|V^0_i|,|W^0_i|\leq M$,\
  $\frac{|V^0_{i+1}-V^0_{i}|}{h},\frac{|W^0_{i+1}-W^0_{i}|}{h}\leq
  L$},\ i\in \Z, && \hspace{-1cm}\text{and\qquad
    $\tau \leq C_2 h^{2s\vee1}f_s(h)$},
  \smallskip
\end{align}
with $f_s(h)=1$ if $s\not=\frac{1}{2}$ and $f_s(h)=\frac{1}{|\log(h)|}$ if $s=\frac{1}{2}$,  $C_1^{-1}=C_s m (2M)^{m-1}$ and $C_2^{-1}=C_s m L^{m-2}  (L\vee 2
M)$. 
Then the solutions $V_i^j$ and $W_i^j$ of \eqref{eq:nums} starting
from $V^0$ and $W^0$ satisfy the following:
\begin{enumerate}[\rm (a)]
  \medskip
\item\label{theo:non} (Nondecreasing) \ $V_{i}^{j}\leq V_{i+1}^{j}$\quad
  $\forall i\in \mathbb{Z}, \ j\in \N$.\medskip
\item\label{theo:st} ($\ell^\infty$-stability)
  $\displaystyle\sup_{i\in \Z, j\in \N}|V_{i}^{j}|\leq \sup_{i\in \Z}|V_{i}^{0}|\leq
  M$.\medskip
\item\label{theo:pos} (Nonnegat ivity) \ If\quad $V^0_i\geq 0$\quad  $\forall i\in \mathbb{Z}$,\quad then\quad  $V_{i}^{j}\geq
  0$\quad $\forall i\in \mathbb{Z},\ j\in \N$.\medskip
\item\label{theo:comp} (Comparison) \ If\quad $V_{i}^0\leq
  W_i^0$\quad  $\forall i\in \mathbb{Z}$,\quad then\quad  $V_{i}^{j}\leq
  W_i^{j}$\quad $\forall i\in \mathbb{Z},\ j\in \N$.\medskip
\item\label{theo:cont} ($\ell^\infty$-contraction)  
$\displaystyle\sup_{i\in \Z, j\in \N} |W_i^j-V_i^j|\leq \sup_{i\in \Z} |W_i^0-V_i^0|$.\medskip
\item\label{theo:lip} (Lipschitz-stability) $\displaystyle\sup_{i\in \Z, j\in \N}
  \frac{|V_{i+1}^j-V_i^j|}{h}\leq \sup_{i\in
    \Z}\frac{|V_{i+1}^0-V_i^0|}{h}$\quad for all\quad $h>0$.
\end{enumerate}
\end{corollary}
Note that under assumption \eqref{lem:mono-item5}, part \eqref{theo:lip} implies that $\sup_{i\in \Z}
  \frac{|V_{i+1}^j-V_i^j|}{h}\leq L$.

\begin{proof}  We first prove \eqref{theo:non} and
  \eqref{theo:st}. Since $V_i^0\leq V_{i+1}^0$, Theorem
  \ref{theo:mono} implies that
  $$V_i^1= \Se[V^0]_i \leq
  \Se[V^0]_{i+1}=V_{i+1}^1.$$
  Then since $-M\leq
  V^0_i\leq M$ and  $\pm M=\Se[\pm M]_{i}$, we conclude from Theorem
  \ref{theo:mono} that 
  $$-M=\Se[- M]_{i}\leq \Se[V^0]_i=V^1_i=\Se[V^0]_i \leq \Se[ M]_{i} =M.$$
  Hence \eqref{theo:non} and
  \eqref{theo:st} hold for $j=1$. The general result then follows by
  iteration/induction. In a similar way we can also prove
  \eqref{theo:pos}. In view of \eqref{theo:non} and
  \eqref{theo:st}, \eqref{theo:comp} follows directly from Theorem
  \ref{theo:mono} by induction. 

For \eqref{theo:cont}, consider 
\[
Z_i^j:= V_i^j+ \sup_{i\in \Z}|W_i^0-V_i^0|.
\]
Since the equation is invariant under translations  of $V$ (Lemma  \ref{lem:mono1} (iii)), $Z_i^j$ solves \eqref{eq:nums}. Moreover,
\[
W_i^0= V_i^0 + (W_i^0 - V_i^0)  \leq V_i^0+ \sup_{i\in \Z}|W_i^0-V_i^0|= Z_i^0.
\]
Since $Z^0$ is nondecreasing and bounded, we can use comparison in
part \eqref{theo:comp},  
\[
W_i^j\leq Z_i^j= V_i^j+ \sup_{i\in \Z}|W_i^0-V_i^0|,\qquad
\text{that is,}
\qquad
W_i^j-V_i^j\leq \sup_{i\in \Z}|W_i^0-V_i^0|.
\]
The property  $W_i^j-V_i^j\geq - \sup_{i\in \Z}|W_i^0-V_i^0|$ follows
similarly and \eqref{theo:cont} follows. 

Finally, \eqref{theo:lip} follows from \eqref{theo:cont} taking $W_i^j=V_{i+1}^j$, a solution of \eqref{eq:nums} with data $W_i^0= V_{i+1}^0$.
\end{proof}

%
%

We also preserve the limits $x\to\pm \infty$ of the initial condition. 

\begin{lemma}[Limits at infinity]\label{lem:limits}
Under the assumptions of Corollary \ref{coro:prese}, $V^0_i\geq0$
for $i\in\Z$,
\[
\lim_{i\to-\infty} V^0_i = 0, \qquad \textup{and} \qquad
\lim_{i\to+\infty} V^0_i = M, 
\]
then for all $j\in \N$,
\[
\lim_{i\to-\infty} V^j_i = 0  \qquad \textup{and} \qquad
\lim_{i\to+\infty} V^j_i = M.
\]
\end{lemma}

\begin{proof}
By  Corollary \ref{coro:prese}, $V_i^j$ is nondecreasing in $i$ and
bounded from above and below by $0$ and $M$. Hence the limits
$\lim_{i\to\pm\infty}V_i^j:=l^{\pm}\in [0,M]$ exist and
\[
0\leq  D_h V^{j}_i \leq \frac{1}{h} (V^{j}_{i+1}-V^{j}_{i-1}).
\]
Consequently, by \eqref{as:m},
$|(-\Delta)_h^{s}V_i^j|\leq 
M C_s {h^{-2s}}$, and
\begin{align*}
&\lim_{i\to\pm\infty}\tau|L_h^s[V_\cdot^j]_i|=\lim_{i\to\pm\infty}
  \tau\big(D_hV_i^j\big)^{m-1}\big|(-\Delta)_h^{s}V_i^j\big|
 \leq 
 M C_s \frac{\tau}{h^{2s+m-1}}\Big( \lim_{i\to\pm\infty}
  (V^{j}_{i+1}-V^{j}_{i-1})\Big)^{m-1}= 0.
\end{align*}
We can therefore conclude that
\[
\begin{split}
\lim_{i\to \pm \infty} V_i^{j+1}&= \lim_{i\to \pm \infty}
\left(V_i^{j} -
\tau L_h^{s}[V_\cdot^j]_i\right)
=\lim_{i\to \pm \infty} V_i^{j}+0.
\end{split}
\]
An iteration in $j$ then completes the proof.
\end{proof}

As a preparation for the convergence proof, we now extend the solution $V$ from the grid to the whole space: Let
$V_h: \overline{Q_T} \to \R$ be the solution of
\begin{equation}
\label{eq:schmBS}
S(h,t, V_h(x,t),V_h(\cdot,t-\tau))=0, \quad (x,t)\in \overline{Q_T},
\end{equation}
where
\begin{align}\label{eq:defSchBS}
&S(h,t, V_h(x,t),V_h(\cdot,t-\tau))\\
&\quad=\left\{ 
\begin{array}{ll}
\displaystyle \frac{V_h(x,t)-V_h(x,t-\tau)}{\tau} 
 - \LL(h,
V_h(x,t-\tau),  V_h(\cdot,t-\tau)), \quad &x\in \R, \  t\in[ \tau,T),\\
\\
V_h(x,t)-v_0(x), \quad &x\in \R, \ t\in[0,\tau),
    \end{array} 
\right.\nonumber
\end{align}
where $\LL$ is defined as in \eqref{eq:defLcal} on the grid
$x+ h\Z$. First note that $V_h$ solves the
scheme in every point in space and not only on the grid. Moreover,
$V_h$ is an extension of $V$ since they coincide on the grid: $V(x_i,t_j)=V_i^j$
for every $i,j$. To see that, restrict \eqref{eq:defSchBS} to the grid  $ h\Z\times \mathcal{T}_{\tau}^T$, to recover trivially \eqref{eq:nums}--\eqref{eq:numic}.
Also note that we omit the dependence on $\tau$ in $S$ since we will
always assume $\tau= o_h(1)$. 

\begin{remark}[Monotonicity]\label{rem:monotonicity}
 The scheme $S$ is monotone:  
  Let  $(x,t)\in \overline{Q_T}$  and  $\psi$ be such that $V_h(x,t)=\psi(x,t)$ and $V_{h}(\cdot,t-\tau)\leq\psi(\cdot,t-\tau)$ if $t\geq \tau$. 
 Then for  $t\in[0,\tau)$, 
\[
S(h,t, V_h(x,t),V_h(\cdot,t-\tau)) =V_h(x,t)-v_0(x) =\psi(x,t)-v_0(x) =S(h,t, \psi(x,t),\psi(\cdot,t-\tau)),
\]
 while for $t\geq \tau$, we use comparison for the one-step propagator $\Se$ in \eqref{def-S} given by Theorem \ref{theo:mono} 
  to get 
\begin{align*}
    S(h,t, V_h(x,t),V_h(\cdot,t-\tau))&= \frac{V_h(x,t)-V_h(x,t-\tau)}{\tau}
     - \LL(h,
V_h(x,t-\tau), V_h(\cdot,t-\tau))\\
&=\frac{\psi(x,t)}{\tau}-\frac{1}{\tau}\underset{\vspace{0.4cm} =\Se [V_h(\cdot,t-\tau)](x)}{\underbrace{\Big(V_h(x,t-\tau)
     +\tau \LL(h,
V_h(x,t-\tau), V_h(\cdot,t-\tau))\Big)}}\\
&\geq \frac{\psi(x,t)}{\tau}-\frac{1}{\tau}\Big(\psi(x,t-\tau)
     + \tau \LL(h,
\psi(x,t-\tau), \psi(\cdot,t-\tau))\Big)\\
&= S(h,t, \psi(x,t),\psi(\cdot,t-\tau)).
\end{align*}

\end{remark}

  We have observed that the scheme is is consistent with equation \eqref{eq:integPp}.
To be precise, 
we 
have the following type of
consistency for the scheme \eqref{eq:defSchBS} (and
\eqref{eq:nums}--\eqref{eq:numic}). 

\begin{lemma}[Consistency]\label{lem:consistency}
Assume $s\in(0,1)$, $m\geq 2$, \eqref{as:cons}, $v_0$ bounded, $\tau
=o(1)$ as $h\to 0$,
and $S$ be defined by \eqref{eq:defSchBS}. If $\phi\in C^2\cap
  C_b(\overline{Q_T})$ 
 and $\eta_h \geq0$ is such that $\eta_h(\frac{1}{h^{2s\vee 1}}+\frac{1}{\tau}) \to 0$ as $h\to0$, then
\begin{equation*}
\begin{split}
&\liminf_{(h,t,x,\xi)\to (0,t_0,x_0,0)} S(h,t, \phi(x,t)+\xi-\eta_h,\phi(\cdot,t-\tau)+\xi),
\\& \quad  \geq \left\{\begin{array}{ll}
\phi_t(x_0,t_0)+ \max\{\phi_x(x_0,t_0),0\}^{m-1} (-\Delta)^{s}\phi (x_0,t_0), \quad & (x_0,t_0)\in Q_T,\\[0.2cm]
\min\left\{\phi(x_0,0)-(v_0)^*(x),\phi_t(x_0,0)+ \max\{\phi_x(x_0,0),0\}^{m-1} (-\Delta)^{s}\phi (x_0,0)\right\}, \quad & (x_0,t_0)\in \R\times\{0\},
\end{array}\right.
\end{split}
\end{equation*}
and
\begin{equation*}
\begin{split}
&\limsup_{(h,t,x,\xi)\to (0,t_0,x_0,0)} S(h,t, \phi(x,t)+\xi+\eta_h,\phi(\cdot,t-\tau)+\xi)
\\& \quad  \leq  \left\{\begin{array}{ll}
\phi_t(x_0,t_0)+ \max\{\phi_x(x_0,t_0),0\}^{m-1} (-\Delta)^{s}\phi (x_0,t_0), \quad & (x_0,t_0)\in Q_T,\\[0.2cm]
\max\left\{\phi(x_0,0)-(v_0)_*(x),\phi_t(x_0,0)+ \max\{\phi_x(x_0,0),0\}^{m-1} (-\Delta)^{s}\phi (x_0,0)\right\}, \quad & (x_0,t_0)\in \R\times\{0\}.
\end{array}\right.
\end{split}
\end{equation*}
\end{lemma}

\begin{proof}
 We only prove the $\limsup$ case  since the $\liminf$ case is similar.
 First observe that, by translation invariance of $\LL$ (Lemma \ref{lem:mono1} (iii)), we have
\begin{align*}
&S(h,t, \phi(x,t)+\xi +
\eta_h,\phi(\cdot,t-\tau)+\xi)\\[0.2cm]
&\quad=\left\{ 
\begin{array}{ll}
\displaystyle \frac{\phi(x,t)-\phi(x,t-\tau)+\eta_h}{\tau}
 -  \LL(h, \phi(x,t-\tau)+ \eta_h,\phi(\cdot,t-\tau)), \quad &x\in \R, \  t\in[ \tau,T),\\
\\
\phi(x,t)-v_0(x)+ \xi+ \eta_h
, \quad &x\in \R, \ t\in[0,\tau).
    \end{array} 
\right.
\end{align*}
Then note that $\frac{\phi-\phi(\cdot,\cdot-\tau)}{\tau}\to
 \phi_t$ locally uniformly as $\tau\to0$. 
Assume first $t_0\in(0,T)$. For small enough  $h$, we then have $t_0\in (\tau,T)$.
Hence by the definitions  and consistency of  $D_h$ and 
$(-\Delta)^{s}_h$ (see  \eqref{as:cons}), 
and the regularity of $\phi$,
\begin{equation*}
\begin{split}
  &\limsup_{(h,t,x,\xi)\to (0,t_0,x_0,0)}
  S(h,t, \phi(x,t)+\xi+\eta_h,\phi(\cdot,t-\tau)+\xi)\\
  & \qquad = \lim_{(h,t,x)\to (0,t_0,x_0)} \Big\{
  \phi_t(x,t)+\tfrac{\eta_h}{\tau}+ \max\big\{\phi_x(x,t)
  \pm\tfrac{\eta_h}{h},0\big\}^{m-1} \Big((-\Delta)^{s}\phi(x,t)+\eta_h\textstyle\sum_{k}\omega_k\Big) + o_h(1)\Big\}\\
& \qquad =  \phi_t(x_0,t_0)+ \max\big\{\phi_x(x_0,t_0),0\big\}^{m-1}  (-\Delta)^{s}\phi (x_0,t_0),
\end{split}
\end{equation*}
where the last step follows by
the localy uniform convergence of $D_h\phi$
and $(-\Delta)_h^{s}\phi$, the fact that $|\sum_k\omega_k|\leq C_sh^{-2s}$ by \eqref{as:m},  and the  assumed  decay rates of $\eta_h$.\footnote{\label{fn} The notation $o_\rho(1)$ is used to denote 
a reminder that goes to zero  $\rho\to0$ and depends only on $s$ and the $C^2$-norm of $\phi$ in a sufficiently large neighborhood of $(x_0,t_0)$.}

Next assume $t_0=0$. Then we can approach $t_0$ by points
from $[\tau, T)$ or by points from $[0,\tau)$ which gives different
    results by the definition of $S$:
\begin{align*}
& \limsup_{(h,t,x,\xi)\to (0,0,x_0,0)} S(h,t, \phi(x,t)+\xi+\eta_h,\phi(\cdot,t-\tau)+\xi)\\
&  \leq  \limsup_{(h,t,x,\xi)\to (0,0,x_0,0)}     \max\bigg\{ \phi_t(x,t)+\tfrac{\eta_h}{\tau}+ \max\big\{\phi_x(x,t)
  \pm\tfrac{\eta_h}{h},0\big\}^{m-1} \Big((-\Delta)^{s}\phi(x,t)+\eta_h\textstyle\sum_{k}\omega_k\Big) + o_h(1),\\
&\hspace{10.5cm} 
\phi(x,t)-v_0(x)+ \xi +\eta_h\bigg\}\\
&\quad = \max\Big\{\phi_t(x_0,0)+ \max\big\{\phi_x(x_0,0),0\big\}^{m-1} 
  (-\Delta)^{s}\phi (x_0,0), \,  \phi(x_0,0)-(v_0)_*( x_0)\Big\}.\qedhere
\end{align*}
\end{proof}

Convergence of the schemes will follow from the arguments of Barles-Perthame-Souganidis using so-called ``half-relaxed limits'' of $V_h$:
\[
\overline{v}(x,t)=\limsup_{(y,s,h)\to (x,t,0^+)} V_h(y,s)\qquad\text{and} \qquad \underline{v}(x,t)=\liminf_{(y,s,h)\to (x,t,0^+)} V_h(y,s),
\]
where $V_h$ is the extension of $V$ defined iteratively by  \eqref{eq:schmBS}
and \eqref{eq:defSchBS}.
Since $V_h$ is continuous in space and c\`adl\`ag/RCLL
in time, and by
Corollary \ref{coro:prese} uniformly bounded, $\overline{v}$ and
$\underline{v}$ are bounded upper and lower semicontinuous functions
respectively. Note that since $V_h$ can have both increasing and decreasing
jumps, it is neither upper nor lower semicontinuous in general. 

By a simple adaptation to our context of the argument of
\cite{BaSo91}, it follows that $\overline{v}$ and $\underline{v}$ are
sub and supersolutions of 
\eqref{eq:integPp}-\eqref{eq:integPini}.

\begin{theorem}[Identification of limits]\label{thm:BS}
Under the assumptions of Corollary \ref{coro:prese}, \eqref{as:cons},
and $v_0$  is bounded,  
it follows that $\overline{v}$ (resp. \underline{v}) is a viscosity subsolution (resp. supersolution) of \eqref{eq:integPp}-\eqref{eq:integPini}. 
\end{theorem}

 We sketch the proof for completeness. 
\begin{proof}
 We only prove the subsolution part, the supersolution part is similar.  Let $\phi\in C^2\cap C_b(\overline{Q_T})$ and $(x_0,t_0)\in \overline{Q_T}$
such that $\overline{v}-\phi$ reaches a  strict  global maximum in
$\overline{Q_T}$, $\overline{v}(x_0,t_0)=\phi(x_0,t_0)$ (and thus
$\overline{v}< \phi$). 
By the definition of $\overline{v}$ there
is a sequence $(y_n,t_n,h_n,\tau_n)\in \overline{Q_T}\times(0,1)^2$ such that,\footnote{ The assumptions of Corollary \ref{coro:prese} includes CFL conditions forcing $\tau\to0$ as $h\to0$.} 
\[
(y_n,t_n,h_n,\tau_n)\to (x_0,t_0,0^+,0^+)\qquad \textup{and}\qquad
V_{h_n}(y_n,t_n)\to \overline{v} (x_0,t_0). 
\]
Moreover,  going to subsequences if necessary,  for every $n$ we can take $(y_n,t_n)$ to be  so  close to the global supremum of $V_{h_n}-\phi$  that 
$$(V_{h_n} - \phi)(y_n,t_n) +\eta_{h_n} > \sup_{\overline{Q_T}}\,(V_{h_n} -
\phi)=: \xi_n$$
 for $\eta_{h_n}\geq0$ satisfying $\eta_{h_n}(\frac{1}{\tau_n}+\frac1{h_n^{\sigma\vee1}})\to0$  
as $n\to \infty$.\footnote{For upper semicontinuous $V_h$, the  $\sup$ is a  $\max$ 
and we can take
$\eta_{h_n}=0$ \cite{BaSo91}, see e.g. Lemma 4.2 and proof in
\cite{Bar94}. Here $V_h$ is not upper semicontinuous and the proof
must be slightly modified,  cf.  \cite{mopost}. 
}
It then follows that 
 $\xi_n \to0$,
\[
 V_{h_n}(y_n,t_n)  > \phi (y_n,t_n)+\xi_n-\eta_{h_n},  \qquad\text{and} \qquad
V_{h_n}(x,t)\leq \phi(x,t) +\xi_n \quad \textup{in}\quad \overline{Q_T}.
\]
Then, by first using monotonicity of the scheme given by  Remark \ref{rem:monotonicity}, 
and then the translation properties of $S$, see \eqref{eq:defSchBS} and Lemma
\ref{eq:defLcal} (iii), we have
\begin{equation*}
\begin{split}
0&=S(h_n,t_n, V_{h_n}(y_n,t_n),V_{h_n}(\cdot,t_n-\tau_n))\\
&\geq S(h_n,t_n, \phi(y_n,t_n)+\xi_n-\eta_{h_n},\phi(\cdot,t_n-\tau_n)+\xi_n).
\end{split}
\end{equation*}
 Thus, taking $\liminf$ and using consistency   (Lemma \ref{lem:consistency}), we find that $\overline{v}$ is a viscosity subsolution of \eqref{eq:integPp}-\eqref{eq:integPini} at $(x_0,t_0)$.
\end{proof}

%

Now give the convergence results, first for uniformly continuous data.
\begin{corollary}[Convergence I]\label{coro:comvinteg1}
Under the assumptions of Corollary \ref{coro:prese}, \eqref{as:cons},
and $v_0\in BUC(\R)$, 
 it  follows that
\[
V_h \to v \quad \textup{as} \quad h\to0 \quad  \textup{locally
  uniformly in} \quad \overline{Q}_T,
\]
where $v$ is the unique viscosity solution  of \eqref{eq:integPp}-\eqref{eq:integPini}.
\end{corollary}
\begin{proof}
By definition $\overline{v}\geq \underline{v}$. By Theorem \ref{thm:BS},
$\overline{v}$ is a viscosity subsolution and $\underline{v}$ is a
viscosity supersolution, and then
$\overline{v}\leq \underline{v}$ by comparison (Theorem
\ref{thm:strongCP}). Consequently, 
\[
v(x,t):=\overline{v}(x,t)=\underline{v}(x,t)= \lim_{(y,s,h)\to (x,t,0^+)} V_h(y,s).
\]
is the unique viscosity solution of
\eqref{eq:integPp}-\eqref{eq:integPini}. Since
$\limsup_{h\to0}V_h\leq \overline{v}=\underline{v}\leq
\liminf_{h\to0}V_h$, $V_h\to v$ pointwise. Uniform convergence of $V_h$
then follows by 
standard arguments, see e.g. Lemma 4.1 in \cite{Bar94}.
\end{proof}

 Then  we present the result for discontinuous data.

\begin{corollary}[Convergence II]\label{coro:comvdirac}
Under the assumptions of
Corollary \ref{coro:prese} and \eqref{as:cons}, if $v_0$ is  bounded,  
and that there is a
discontinuous viscosity solution $\mathcal{V}$  of
\eqref{eq:integPp}-\eqref{eq:integPini}, then
\[
V_h(x,t)\to \mathcal{V}(x,t) \quad\text{as} \quad h\to 0  \qquad
\textup{for all} \qquad (x,t)\in \mathcal C_\mathcal{V},
\]
where $\mathcal C_\mathcal{V}$ is the set of points where
  $\mathcal{V}$ is continuous. Moreover, the convergence is locally
uniform in any ball contained in $\mathcal C_\mathcal{V}$.
\end{corollary}

\begin{proof}
As before,  $\overline{v}\geq \underline{v}$ and $\overline{v}$ and
$\underline{v}$ are viscosity sub and supersolutions of
\eqref{eq:integPp}-\eqref{eq:integPini}. Let $\mathcal V$ be continuous
at $(x_0,t_0)$. By comparison, Theorem \ref{thm:comppp2},
\[
\overline{v}(x_0,t_0)\leq (\mathcal{V}_*)^*(x_0,t_0)=\mathcal{V}=(\mathcal{V}^*)_*(x_0,t_0)\leq \underline{v}(x_0,t_0).
\]
Pointwise and then uniform convergence follows as in the proof of
Corollary \ref{coro:comvinteg1}.
\end{proof}

We will now prove the first of the main convergence results of the paper.

\begin{proof}[Proof of Theorem \ref{thm:main}] 
(a) Since $\mu_0\in L^1(\R)$, $v_0\in BUC(\R)$ and by Corollary
\ref{coro:comvinteg1}, $V_h\to v$ locally uniformly in
$\overline{Q}_T$ as $h\to0$,  where $v$ is the unique viscosity solution of \eqref{eq:integPp}-\eqref{eq:integPini}. By Lemma \ref{lem:incr}, $v$ is nondecreasing, and then by Lemma \ref{lem:coincide} it is a viscosity solution of \eqref{eq:integP}-\eqref{eq:integPini} (which is unique by comparison, Theorem \ref{thm:strongCP}).  Here $V_h$ is defined via the scheme
posed in the whole space, while $\overline{V}_h$ is the interpolation of the
scheme defined on the grid. By Theorems \ref{thm:l1linfv}, $v\in BUC$,
and then for $(x,t)\in[x_{i-1},x_i]\times[t_j,t_{j+1})$,\footnote{ The $o_h(1)$ notation  
denotes 
a reminder that goes to zero  as $h\to0$ and depends only on the modulus of continuity of  $v$.} 
\[
\begin{split}
  |\overline{V_h}(x,t)-v(x,t)|
&\leq  \frac{x_i-x}{h} |V_{i-1}^j -v(x_{i-1},t_j)| + \frac{x-x_{i-1}}{h} |V_{i}^j -v(x_{i},t_j)| + o_h(1)\\
&\leq  \frac{x_i-x}{h} |V_h(x_{i-1},t_j) -v(x_{i-1},t_j)| + \frac{x-x_{i-1}}{h} |V_h(x_{i},t_j) -v(x_{i},t_j)|+ o_h(1)\\
& \leq \sup_{(y,\rho)\in[x_{i-1},x_i]\times[t_j,t_{j+1})} |V_h(y,\rho)-v(y,\rho)|+o_h(1).
\end{split}
\]
This implies that $\overline{V}_h\to v$ as $h\to0$ locally uniformly in $\overline{Q_T}$.
\medskip

\noindent (b) By Theorem \ref{thm:v}, $v\in BUC (\overline{Q}_T\setminus B_{\veps}(0,0))$ for all $\veps>0$. Thus by Corollary \ref{coro:comvdirac}, 
\[
v(x,t)=  \lim_{(y,s)\to (x,t)} V_h(y,s) \quad \textup{for all} \quad (x,t)\in \overline{Q}_T\setminus B_{\veps}(a,0)).
\]
This implies that $V_h\to v$ locally uniformly in
$\overline{Q}_T\setminus B_{\veps}(a,0)$ as $h\to0$. Arguing as in the
proof of part (a), we conclude the same local uniform convergence for $\overline{V}_h$.
\end{proof}

\section{Analysis of the discretization of the fractional pressure
  equation
}\label{sec:propUandconv}
 Recall that $\mu_0\in \mathcal{M}^{+}(\R)$,  $s=1-\sigma$ , and  $V_i^j=V(x_i,t_j)$
solves the scheme for 
\eqref{eq:integP}-\eqref{eq:integPini}, 
\begin{equation}\label{eq:nums22}\tag{S}
V_i^{j+1}=V_i^j + \tau L_h V_i^{j}\qquad\text{for} \qquad i\in\Z, \quad j\in \N,
\end{equation}
with $V_i^0= v_0(x_i):= \int_{-\infty}^{ x_i} \dd \mu_0(y)$, and that the
approximation $U$ of \eqref{eq:maineq}-\eqref{eq:maineqini} is defined
from $V$ by
\begin{equation}\label{eq:relUV}
U_i^j= \frac{V_i^j-V_{i-1}^j}{h} \qquad \textup{or equivalently} \qquad  V_i^j= h \sum_{k=-\infty}^i U_k^j\qquad\text{for} \qquad i\in\Z, \quad j\in \N.
\end{equation}
 We also recall the definition of the interpolants $\overline{U}_h:\overline{Q}_T\to\R$ and
$\overline{V}_h:\overline{Q}_T\to\R$,  
\begin{align}\label{eq:interpU}
\overline{U}_h(x,t)&=U_i^j \qquad \textup{if} \qquad
(x,t)\in[x_{i-1},x_i)\times[t_j,t_{j+1}), \\[0.2cm]
\label{eq:interpV}
\overline{V}_h(x,t)& = \frac{x_i-x}{h} V_{i-1}^j + \frac{x-x_{i-1}}{h} V_i^j  \qquad \textup{if} \qquad (x,t)\in[x_{i-1},x_i)\times[t_j,t_{j+1}).
\end{align}
The relation \eqref{eq:relUV} allows us to inherit properties from Corollary \ref{coro:prese} and
Lemma \ref{lem:limits} to $U_h$ and $\bar U_h$. 

\begin{corollary}\label{cor:propschu}
Assume $s\in(0,1)$, $m\geq2$, $h,\tau>0$, \eqref{as:m}, and
either \eqref{as:CFL1} or \eqref{as:CFL2}. Then 
\begin{enumerate}[\rm (a)]
\item (Positivity) \ $U_{i}^j \geq 0$ \ for $i\in \Z$, $j\in \N$,\quad
  and\quad $\overline{U}_h\geq0$ \ in $\overline{Q}_T$.
\item\label{cor:propschu-itemCons} (Conservation of mass) \ $\displaystyle h\sum_{i=-\infty}^\infty
  U_i^j = \int_{\R}\overline{U}_h(x,t)\,\dd  x =M$\quad for\quad $j\in \N$, \ $t\in
  [t_j,t_{j+1})$. 
\item (Maximum principle) If $\mu_0\in L^\infty(\R)$, then $\displaystyle \sup_{i\in \Z, j\in \N} |U_i^j|=\|\overline{U}_h\|_{L^\infty(Q_T)}\leq\|\mu_0\|_{L^\infty(\R)} $.
\end{enumerate}
\end{corollary}

\begin{proof}
(a) Since $V_i^j\leq V_{i+1}^j$ for all $i\in \Z$ and all $j\in \N$, it
  is clear that $U_i^j\geq0$. The result for $\overline{U}_h\geq0$
  follows by interpolation. (b) For $t\in [t_j,t_{j+1})$ and
    definition of $\bar U$,
\[
\int_{\R}\overline{U}_h(x,t)\dd x= \sum_{i=-\infty}^\infty \int_{x_{i-1}}^{x_i}\overline{U}_h(x,t)\dd x = \sum_{i=-\infty}^\infty \int_{x_{i-1}}^{x_i}U_i^j\dd x = h \sum_{i=-\infty}^\infty U_i^j.
\]
Then we conclude by the relation to $V$ and Lemma \ref{lem:limits},
\[
\sum_{i=-\infty}^\infty U_i^j= \lim_{k\to+\infty}\sum_{i=-\infty}^k
U_i^j=\lim_{k\to+\infty} \frac1hV_k^j=\frac Mh.
\]
(c) Follows from Corollary \ref{coro:prese} (f):
\[
\begin{split}
\|\overline{U}_h\|_{L^\infty(Q_T)}&=\sup_{i\in \Z, j\in \N} |U_i^j|= \frac{1}{h} \sup_{i\in \Z, j\in \N} |V_i^j-V_{i-1}^j| \leq \frac{1}{h} \sup_{i\in \Z} |V_i^0-V_{i-1}^0|\\
&=  \frac{1}{h}  \sup_{i\in \Z} |\int_{-\infty}^{x_i} \mu_0(y) \dd y-\int_{-\infty}^{x_{i-1}} \mu_0(y) \dd y|\leq \frac{1}{h} \|\mu_0\|_{L^\infty(\R)} \sup_{i\in \Z} \int_{x_{i-1}}^{x_i}\dd y= \|\mu_0\|_{L^\infty(\R)}.\quad\qedhere
\end{split}
\]
\end{proof}

We now prove our main convergence results, that $\overline U_h$
converges 
weakly-$*$ in $C_b(\R)$ to the solution $u$ of
\eqref{eq:maineq}-\eqref{eq:maineqini}. This will mostly  come as a
consequence of the convergence results for $\overline V_h$.

\begin{proof}[Proof of Theorem \ref{thm:main2}]
  We first prove part (a).
\medskip

\noindent\emph{Step 1:} First we check that $\overline{V}_h(x,t)=\int_{-\infty}^x \overline{U}_h(y,t) \dd y$ (see  Remark \ref{rem:intutov}). For $(x,t)\in [x_{i-1},x_i)\times [t_{j-1},t_j)$,
\[\begin{split}
\overline{V}_h(x,t)&= \frac{x_i-x}{h} V_{i-1}^j + \frac{x-x_{i-1}}{h} V_i^j = (x_i-x) \sum_{k=-\infty}^{i-1} U_k^j +(x-x_{i-1})\sum_{k=-\infty}^{i} U_k^j\\
&= (x-x_{i-1}) U_{i}^j + h \sum_{k=-\infty}^{i-1} U_k^j=
\int_{x_{i-1}}^x \overline{U}(y,t) \dd y + \int_{-\infty}^{x_{i-1}}
\overline{U}(y,t) \dd y. 
\end{split}
\]
Since $\overline{V}_h$ is piecewise linear, the above result implies $\partial_x\overline{V}_h(x,t)= \overline{U}_h(x,t) $ for $t\in[0,T]$ and a.e. $x\in \R$.
\medskip

\noindent\emph{Step 2:}  Let $\varphi \in Lip_{1,1}$
be such that $\supp \partial_x \varphi \subset [-R,R]$  for
some $R>0$.  By conservation of mass (Corollary
\ref{cor:propschu}\eqref{cor:propschu-itemCons}),
\[
\int_{-\infty}^x (\overline{U_h}(y,t)-u(y,t))\dd y=\int_{-\infty}^x \overline{U_h}(y,t)\dd y-\int_{-\infty}^x u(y,t)\dd y \left\{\begin{split}&\stackrel{x\to+\infty}{\longrightarrow}\|u_0\|_{L^1(\R)}-\|u_0\|_{L^1(\R)}=0,\\
&\stackrel{x\to-\infty}{\longrightarrow} 0-0=0.
\end{split}\right.
\]
Hence by integration by parts, the relation between $\overline{U_h}$
and  $\overline{V_h}$,  and the properties of $\varphi$, we find that
\begin{equation*}
\begin{split}
\sup_{t\in[0,T]}\left|\int_{\R} (\overline{U_h}(x,t)-u(x,t))\varphi(x)\dd x\right|&= \sup_{t\in[0,T]} \left|\int_{\R} \left(\int_{-\infty}^x (\overline{U_h}(y,t)-u(y,t))\dd y\right) \partial_x\varphi(x)\dd x\right|\\
&=\sup_{t\in[0,T]}  \left|\int_{\R}( \overline{V_h}(x,t)-v(x,t))\partial_x \varphi(x)\dd x\right|\\
&\leq |\supp\{\partial_x\varphi\}|\|\partial_x
\varphi\|_{L^\infty(\overline{Q_T})} \sup_{(x,t)\in
   \supp\{\partial_x\varphi\}\times[0,T] } |\overline{V_h}(x,t)-v(x,t)|\\
&\leq 2 R T\|\partial_x
\varphi\|_{L^\infty(\overline{Q_T})}  \omega_{R}(h),
\end{split}
\end{equation*}
where
$$\omega_{R}(h):=\sup_{(x,t)\in B(0,R)\times[0,T]}
|\overline{V_h}(x,t)-v(x,t)|\to 0 \qquad\text{as}\qquad h\to0,$$
by local uniform
convergence of $\overline{V_h}$ to $v$ (Theorem \ref{thm:main}).
\smallskip

\noindent\emph{Step 3: } For $R>1$, let $\chi_R\in Lip_{1,1}$
be such that $0\leq \chi_R\leq 1$, 
$\chi_R(x)=1$ if $|x|\geq R$, and $\chi_R(x)=0$ if $|x|\leq R-1$.  
Since $\varphi(x):=\chi_R(x)$ has $\supp
\partial_x\varphi \subset [-R,R]\setminus [-R+1,R-1]$,  by step 2 we get
\begin{equation*}
\begin{split}
  \sup_{t\in[0,T]} \int_{\R} \overline{U}_h(x,t) \chi_R(x)\dd x
&\leq \sup_{t\in[0,T]} \left| \int_{\R} (\overline{U}_h(x,t)-u(x,t)) \chi_R(x)\dd x \right| + \sup_{t\in[0,T]} \int_{\R} u(x,t) \chi_R(x)\dd x\\
&\leq  2RT\,\omega_R(h)+ \sup_{t\in[0,T]} \int_{|x|>R-1} u(x,t) \dd x.
\end{split}
\end{equation*}

\noindent\emph{Step 4: } Assume $\varphi \in Lip_{1,1}$ is
arbitrary. Then $\varphi(1-\chi_R), \varphi\chi_R\in Lip_{1,2}$ (the
Lipschitz constants are bounded by 2), and $\supp
\varphi(1-\chi_R)\subset [-R,R]$. Recall that $\overline{U}_h,u\geq 0$
and $0\leq \chi_R\leq 1$. 
For $R>1$ we then have by adding and subtracting terms, using step 2
(with Lipschitz constant 2),
and then step 3,
\[
\begin{split}
&\sup_{t\in[0,T]}  \Big|\int_{\R} (\overline{U}_h(x,t)-u(x,t))\varphi(x)\dd x\Big| \\
&\quad \leq \sup_{t\in[0,T]} \Big|\int_{\R}
  (\overline{U}_h(x,t)-u(x,t))\varphi(x)(1-\chi_R(x))\dd x\Big|
  +\sup_{t\in[0,T]}\big|\int_{\R}
  (\overline{U}_h(x,t)-u(x,t))\varphi(x)\chi_R(x) \dd x\Big| \\
&\quad \leq 4RT\omega_R(h)
  +\sup_{t\in[0,T]}\int_{\R}
  (\overline{U}_h(x,t)+u(x,t)) \chi_R(x)\dd x \\  
&\quad \leq  6RT\omega_R(h) +2
  \sup_{t\in[0,T]}\int_{|x|>R-1} |u(x,t)| \dd x.
\end{split}
\]
By tightness  (Theorem
\ref{thm:maineq} \eqref{thm:maineq-itemtight}), the last term is
bounded by some modulus $\tilde\omega_u(\frac1R)$. Then by
the definition of the $d_0$ distance, it  follows that
$$\sup_{t\in[0,T]}d_0(\overline{U}_h(\cdot,t),u(\cdot,t))\leq 6RT\omega_R(h) +2\tilde\omega_u(\frac1R). $$
Convergence follows by first taking $R$ large enough to control
the second term and then $h$
small enough to control the first term. 
This completes the proof of (a).

To prove (b) we need to handle
the point mass at $x=a$ separately. Arguing as in (a) we find that
\[
\begin{split}
&\sup_{t\in[0,T]}  \Big|\int_{\R} (\overline{U}_h(x,t)-u(x,t))\varphi(x)\dd x\Big| \\
& \quad \leq    6RT\omega_{r,R}(h) +  \sup_{t\in[0,T]}  \Big|\int_{B(a,r)} (\overline{U}_h(x,t)-u(x,t))\varphi(x)\dd x\Big|+2
  \sup_{t\in[0,T]}\int_{|x|>R-1} |u(x,t)| \dd x\\
\end{split}
\]
where the second term is bounded by  $4Mr$ when $\varphi\in Lip_{1,1}$ and 
$$\omega_{r,R}(h):=\sup_{(x,t)\in B(0,R)\setminus B(a,r)\times[0,T]}
|\overline{V_h}(x,t)-v(x,t)|\to 0 \qquad\text{as}\qquad h\to0.$$
Convergence follows by first taking $r$ small and $R$ large enough to control
the second and third term and then $h$ small enough
to control the first term. 
\end{proof}

\section{Numerical simulations}\label{sec:num}

The simulations of this paper are done with the \emph{powers of the
  discrete Laplacian} method introduced in \cite{Cia-etal18} and
proved to be consistent of order $O(h^2)$ (independently of $s$) in
\cite{dTEnJa18}, see Appendix \ref{app:disfl} for the details. The
scheme \eqref{eq:nums} 
is still just $O(h)$ because of the discretization of first derivatives (see \eqref{eq:discop} and \eqref{eq:discDer}). We use 
compactly supported initial data $\mu_0$. Then, by the finite speed of
propagation property (see Theorem
\eqref{thm:maineq}\eqref{thm:maineq-itemFSP}), there is an
$M>0$ such that $v(x,t)=0$ when
$x<-M$ and $v(x,t)=\int_{\R}u(y,t)\dd y=\mu_0(\R^d)$ when $x>M$ for
$t\in[0,T]$. This allow us to perform simulations in the bounded
domain $|x|<M$ by extending the numerical
solution $V_i^j$ accordingly. The discretization of $(-\Delta)^s$
given in \eqref{eq:discretization} can then be reduced to a finite
sum, and since our scheme is explicit, $(-\Delta)^s_h V^j_i$ and
$D_hV_i^{j}$ are computed based on known $V$-values from the previous
time step.  

\subsection{Explicit solutions ($m=2$) and error measures} 

Let $t_0\geq0$ and $R>0$. By Theorem 2.2 in \cite{BiImKar11}, a
family of explicit solutions of \eqref{eq:maineq}-\eqref{eq:maineqini} when
$m=2$ is given by
\begin{equation}\label{eq:explicitsol}
u(x,t)=k_s(t+t_0)^{-\frac{1}{1+2s}}(R^2-|x(t+t_0)^{-\frac{1}{1+2s}}|^2)_+^s \quad \textup{with} \quad k_s=\frac{1}{2^{2s}(1+2s)}\frac{\Gamma(\frac{1}{2})}{\Gamma(1+s)\Gamma(\frac{1}{2}+s)}.
\end{equation}
It is also standard to check, by a simple change of variables, that
$u$ has the mass
\[
M_{R,s}:=\|u(\cdot,t)\|_{L^1(\R)}=k_s \int_{-R}^R (R^2-y^2)^s \dd y= k_sR^{1+2s} \int_{-1}^1 (1-z^2)^s \dd z= k_s R^{1+2s} \frac{\Gamma(\frac{1}{2}) \Gamma(1+s)}{\Gamma(\frac{3}{2}+s)}.
\]
In particular, when $t_0=0$, the solution has initial condition
\(
\mu_0=M_{R,s} \delta_0.
\)

We compute the error of our approximation at $t=1$ using a relative
$L^\infty$ norm for \eqref{eq:integP}-\eqref{eq:integPini},
\[
E_v(h)=\frac{\sup_{x_i\in h\Z}|v(x_i,1)-\overline{V}_h(x_i,1)|}{\|v(\cdot,1)\|_{L^\infty(\R)}},
\]
while for \eqref{eq:maineq}-\eqref{eq:maineqini} we
use a relative $L^1$ norm and a relative error of second moments, 
\begin{align*}
E_u(h)&=\frac{h \sum_{x_i\in h\Z} |\overline{U_h}(x_i,1)-u(x_i,1)|}{h
  \sum_{x_i\in h\Z} u(x_i,1)} 
\qquad\text{and}\qquad E_u^{w}(h)&=\frac{h\Big|\sum_{x_i\in
  h\Z}(\overline{U_h}(x_i,1)-u(x_i,1))x_i^2\Big|}{\sum_{x_i\in
  h\Z}u(x_i,1)x_i^2}. 
\end{align*}
Since test solutions have compact support, they have moments of every
order, and convergence of moments is a commonly used estimate
for weak convergence. Note that 
\[\int_\R
(W-w)(x)|\frac xR|^2  \dd  x\leq d_0(W,w)=\underset{\varphi\in Lip_{1,1}}{\sup}\int_\R
(W-w)(x)\varphi(x)  \dd  x\leq \|W-w\|_{L^1},\]
 for 
integrable functions $w,W $ 
with $\mathrm{supp}
(w,W)\subset B(0,R)$.  
In each convergence plot below we include a
dashed line showing first order convergence, the formal order of the scheme.

\subsection{Experiment 1:} We take the solution given by
\eqref{eq:explicitsol} with $t_0=1$ and $R=0.5$ for values $s=0.25$,
$s=0.5$ and $s=0.75$. Since $u_0$ is bounded, we use \eqref{as:CFL1}
to choose $\tau$ in terms of $h$. 
In Figure \ref{fig:Error1} we present the errors.

\begin{figure}[h!]
			\begin{center}
				\includegraphics[width=\textwidth]{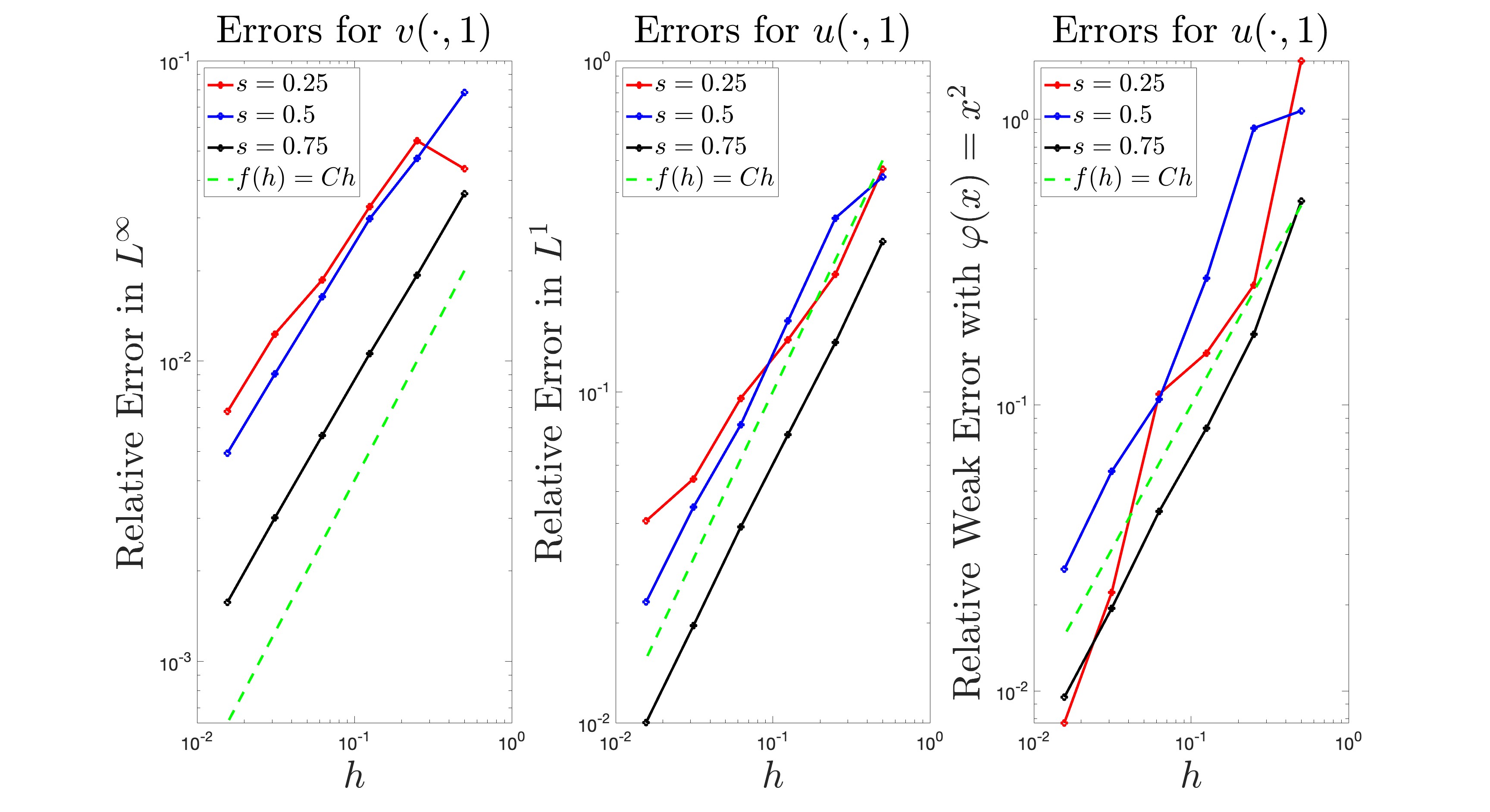}
			\end{center}
\caption{Errors in experiment 1 measured in $E_v(h)$, $E_u(h)$,
  $E_u^w(h)$ respectively.}
\label{fig:Error1}
\end{figure}

\subsection{Experiment 2:} We take the solution given by \eqref{eq:explicitsol} with $t_0=0$ and $R=1$ for values $s=0.25$, $s=0.5$ and $s=0.75$. Note that in this case the initial condition is a delta function given by
\[
\mu_0=M_{R,s} \delta_0.
\]
We then we use \eqref{as:CFL2} to choose $\tau$ in terms of $h$. The
error in \eqref{eq:integP}-\eqref{eq:integPini} is computed as in the
previous example. In Figure \ref{fig:Error2} we present the errors.

\begin{figure}[h!]
			\begin{center}
				\includegraphics[width=\textwidth]{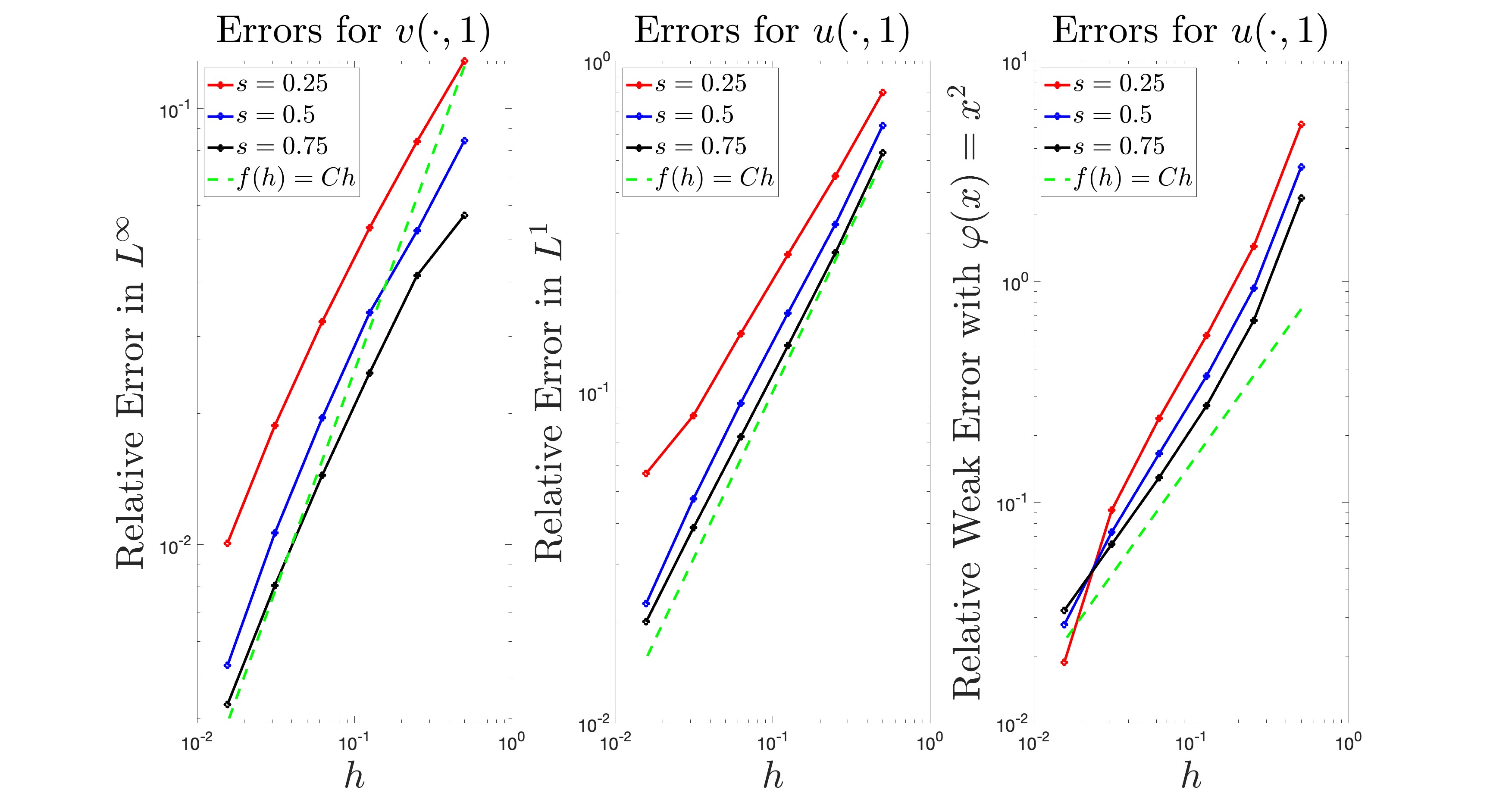}
			\end{center}
\caption{Errors in experiment 2 measured in $E_v(h)$, $E_u(h)$,
  $E_u^w(h)$ respectively.}
\label{fig:Error2}
\end{figure}

\subsection{Experiment 3:} We include here a simulation in the case $m=4$. Since, up to our knowledge, no explicit solutions are known in this case, we use  as  reference a simulation done in a very refined grid in order to compute the numerical errors. More precisely, we consider the initial condition
\[
u_0(x)=e^{-\frac{1}{\left(1-\left(x-3/2\right)^2\right)_+}} + 2e^{-\frac{1}{(1-(x+3/2)^2)_+}}. 
\]
To compute the reference solution we use the numerical scheme with $h_0=\frac{1}{2^{11}}$. The error in \eqref{eq:integP}-\eqref{eq:integPini} is calculated  by
relative $L^\infty$ norm \(E_{\overline{V}_{h_0}}(h)\), while for the
error in \eqref{eq:maineq}-\eqref{eq:maineqini} we use the relative
$L^1$ norm \(E_{\overline{U}_{h_0}}(h)\).  In Figure \ref{fig:Error3}
we present the errors. 

\begin{figure}[h!]
			\begin{center}
				\includegraphics[width=\textwidth]{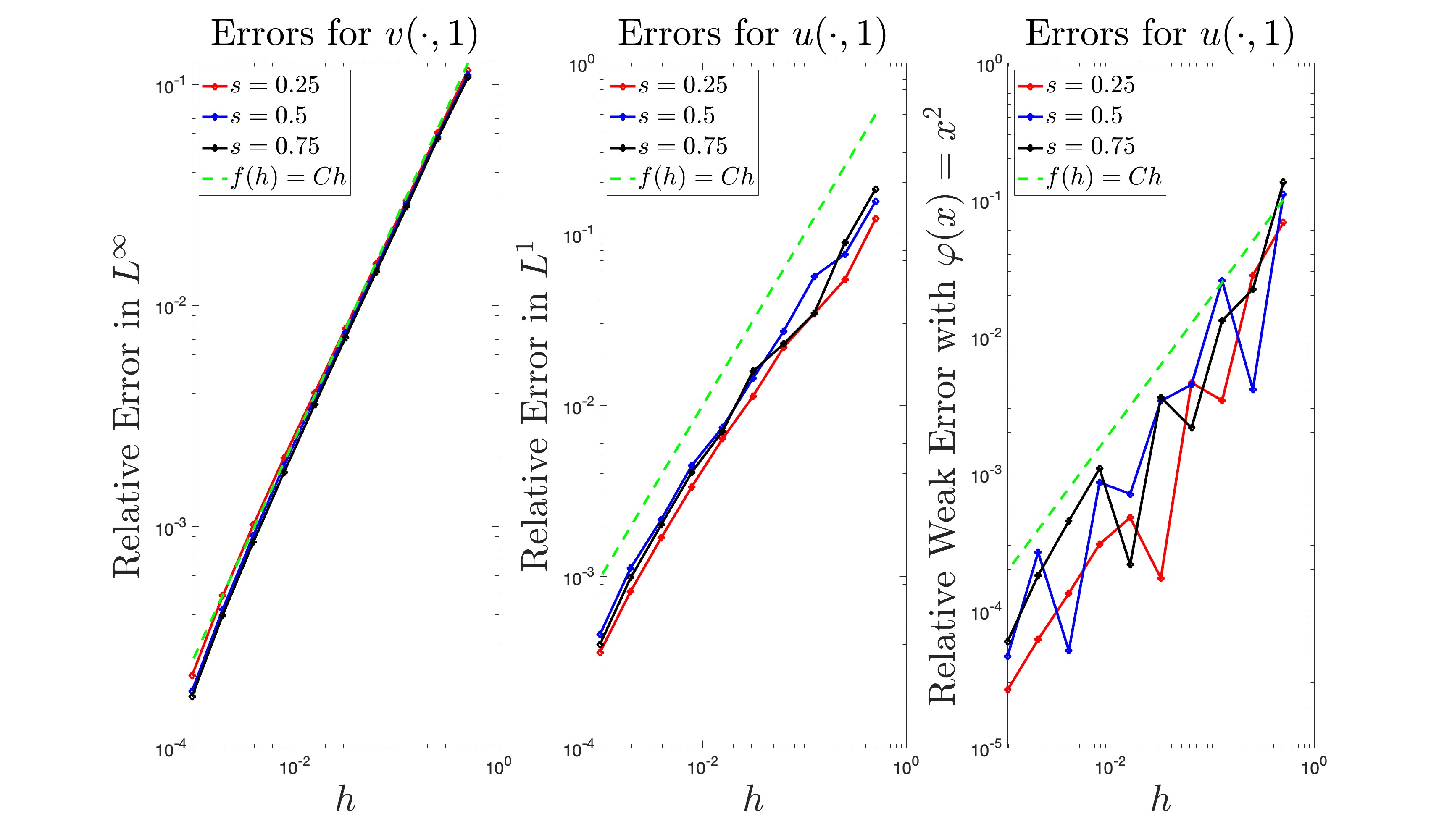}
			\end{center}
\caption{Errors in experiment 3  measured in $E_v(h)$, $E_u(h)$,
  $E_u^w(h)$ respectively.}
\label{fig:Error3}
\end{figure}

\subsection{Experiment 4:} It is well-known that the comparison
principle does not hold for \eqref{eq:maineq}-\eqref{eq:maineqini}. In
Figure \ref{fig:Nocomparison1} we illustrate this by
showing the evolution of two solutions which are ordered initially but
not at later times. We let $u$ be given by \eqref{eq:explicitsol}, take
$t_0=1$ and $R=0.5$, define
\[
(u_{0})_1(x)=u(x-1,0)\qquad \textup{and}\qquad  (u_0)_2(x)=u(x-1,0) + 2\cdot u(x+1,0), 
\]
and compute the corresponding solutions $u_1$ (represented in  dotted  blue)
and $u_2$ (in  solid  red). 

\begin{figure}[h!]
			\begin{center}
				\includegraphics[width=\textwidth]{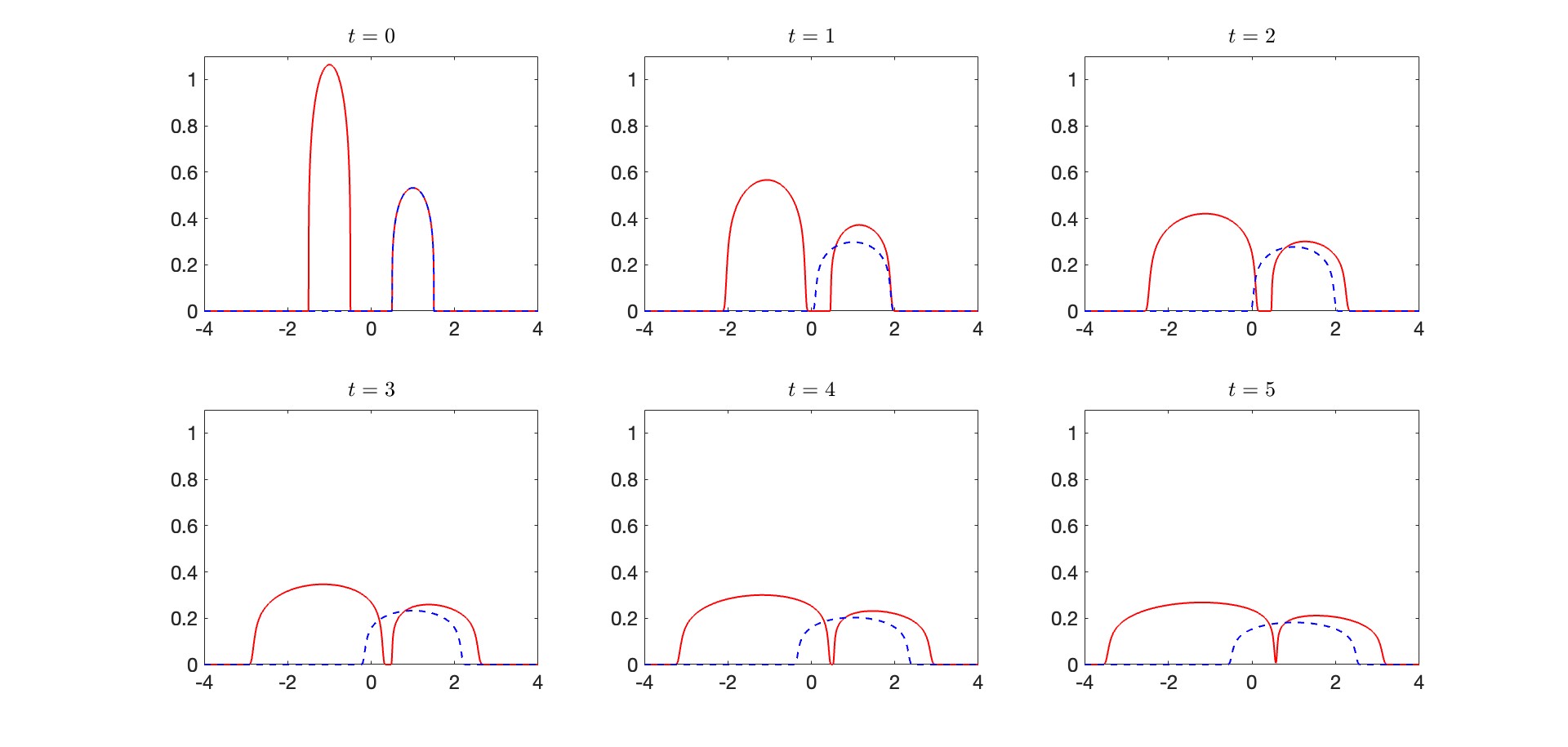}
			\end{center}
\caption{Example of no comparison principle: $u_1$ (dotted  blue) and $u_2$ (solid  red).}
\label{fig:Nocomparison1}
\end{figure}

In Figure \ref{fig:Nocomparison2} we present corresponding solutions of \eqref{eq:integP}-\eqref{eq:integPini}. In this case comparison holds, as expected from the theoretical results.

\begin{figure}[h!]
			\begin{center}
				\includegraphics[width=\textwidth]{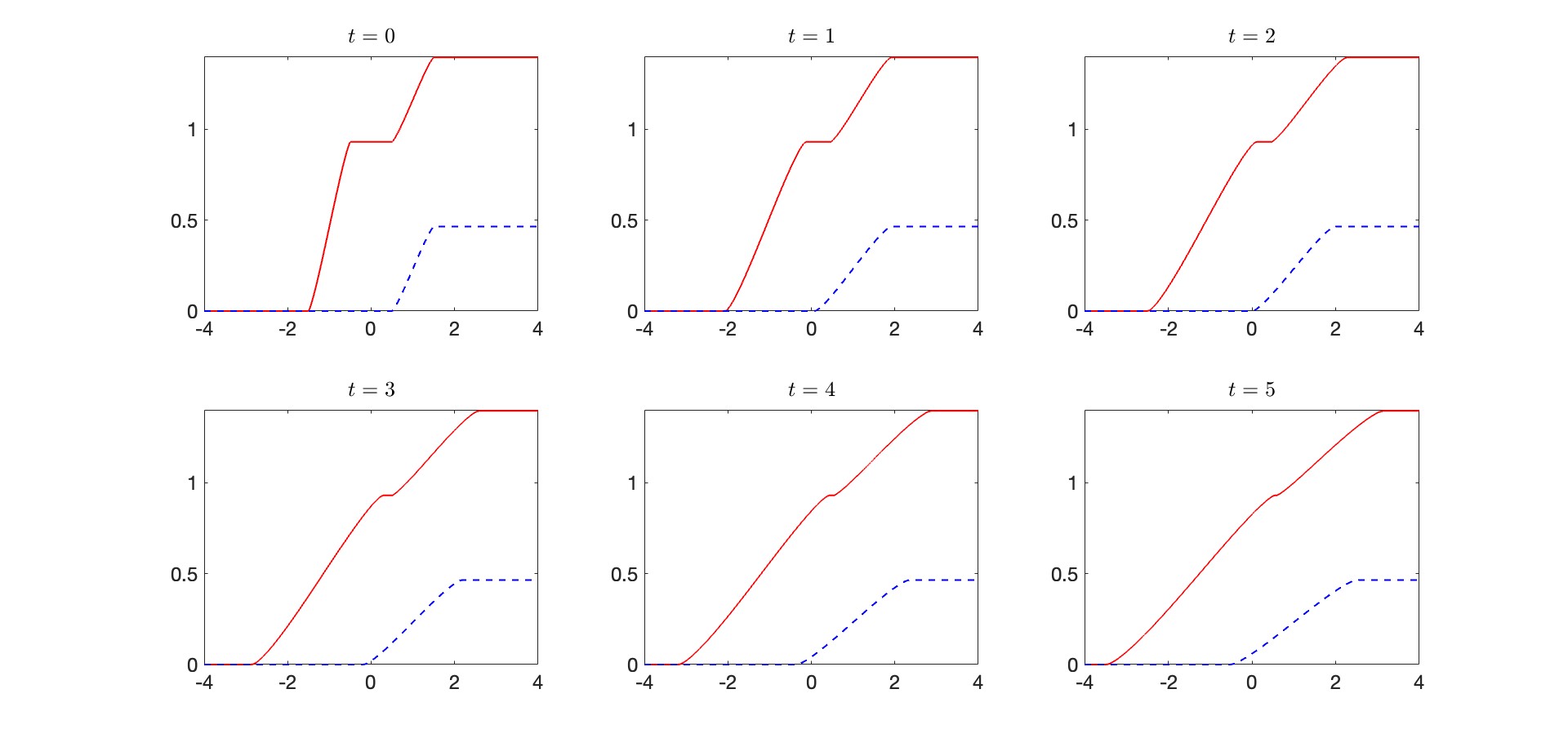}
			\end{center}
\caption{ Example of comparison for the integrated problem: $v_1$ (dotted  blue) and $v_2$ (solid  red).}
\label{fig:Nocomparison2}
\end{figure}

\subsection*{Acknowledgment:}  F. del Teso was supported by the Spanish Government through RYC2020-029589-I, PID2021-127105NB-I00,  CNS2024-154515 and CEX2019-000904-S  funded by the MICIN/AEI.
Part of this material is based upon work supported by the Swedish Research Council under grant no. 2016-06596 while F. del Teso was in residence at Institut Mittag-Leffler in Djursholm,
Sweden, during the research program “Geometric Aspects of Nonlinear Partial Differential Equations”,
fall of 2022.

E. R. Jakobsen received funding from the Research Council of Norway under grant agreement no. 325114 “IMod. Partial differential equations, statistics and data: An interdisciplinary approach to data-based modelling”.

We would like to thank David G\'omez Castro for reading the paper carefully, for his suggestions, and the many good discussions we had. 

\bigskip

\appendix
\section{Discretizations of the fractional Laplacian}\label{app:disfl}

For a reference of various discretizations of the fractional Laplacian we refer to Section 4 in \cite{dTEnJa18}. 

The simulations of this paper are done with the so-called \emph{powers of the discrete Laplacian}, that was introduced in \cite{Cia-etal18} and proved to be consistent of order $O(h^2)$ (unconditionally on $s$) in \cite{dTEnJa18}. The weights of the discretization are given by
\[
\omega_k= \frac{1}{h^{2s}} \frac{2^{2s} \Gamma(1/2+s) \Gamma(|k|-s)}{\sqrt{\pi} |\Gamma(-s)|\Gamma(|k|+1+s) }.
\]
The fact that assumption \eqref{as:cons} is satisfied is proved in Theorem 1.7 in \cite{Cia-etal18}. To check assumption \eqref{as:m} we rely on the following estimate that can be found in Theorem 1.1. in \cite{Cia-etal18}:
\[
\omega_k\leq \frac{c_s}{h^{2s}|k|^{1+2s}} \quad \textup{for some constant} \quad c_s>0.
\]
Thus,
\[
\sum_{k\not=0} \omega_k\leq \frac{c_s}{h^{2s}} \sum_{k\not=0} \frac{1}{|k|^{1+2s}}= C_{s,1} h^{-2s}.
\]
We also have
\[
\sum_{|kh|>1} \omega_k \leq c_s h \sum_{|kh|>1} \frac{1}{|kh|^{1+2s}}\leq c_s \int_{|x|>1/2} \frac{\dd x}{|x|^{1+2s}}= C_{s,2}.
\]
Moreover, if $s>1/2$,
\[
\sum_{0<|kh|\leq 1}|kh| \omega_k \leq c_s  \sum_{0<|kh|\leq 1} \frac{|kh|}{h^{2s}|k|^{1+2s}}=c_s h^{1-2s} \sum_{0<|kh|\leq 1} \frac{1}{|k|^{2s}} = c_s h^{1-2s} \sum_{k\not=0} \frac{1}{|k|^{2s}}=C_{s,3} h^{1-2s},
\]
while, if $s<1/2$,
\[
\sum_{0<|kh|\leq 1}|kh| \omega_k =  c_s h\sum_{0<|kh|\leq1} \frac{1}{|kh|^{2s}}\leq c_s  \int_{-2}^2 \frac{\dd x}{|x|^{2s}}= C_{s,3}.
\]
Finally, if $s=1/2$,
\[
\sum_{0<|kh|\leq 1}|kh| \omega_k \leq c_s\sum_{0<|k|\leq1/h}\frac{1}{|k|}= c_s(1+\int_{1}^{1/h} \frac{\dd x}{x})=c_{s}(1+\log(1/h))\leq C_{s,3}|\log(h)|.
\]
It is enough to take $C_s=\max\{C_{s,1},C_{s,2},C_{s,3}\}$ to ensure that \eqref{as:m} is satisfied.

\bibliographystyle{abbrv}



\begin{thebibliography}{10}


\bibitem{BCMS20}
  R. Bailo, J. A. Carrillo, H. Murakawa, and M. Schmidtchen. 
  \newblock Convergence of a Fully Discrete and Energy-Dissipating
  Finite-Volume Scheme for Aggregation-Diffusion Equations.
  \newblock {\em Math. Models Methods Appl. Sci.} 30(13):2487–2522,
  2020.

  
\bibitem{Bar94}
G.~Barles.
\newblock {\em Solutions de viscosit\'e des \'equations de
  {H}amilton-{J}acobi}, volume~17 of {\em Math\'ematiques \& Applications
  (Berlin) [Mathematics \& Applications]}.
\newblock Springer-Verlag, Paris, 1994.

\bibitem{BaSo91}
G.~Barles and P.~E. Souganidis.
\newblock Convergence of approximation schemes for fully nonlinear second order
  equations.
\newblock {\em Asymptotic Anal.}, 4(3):271--283, 1991.

\bibitem{BeFrOb14}
J.-D. Benamou, B.~D. Froese, and A.~M. Oberman.
\newblock Numerical solution of the optimal transportation problem using the
  {M}onge-{A}mp\`ere equation.
\newblock {\em J. Comput. Phys.}, 260:107--126, 2014.


\bibitem{BF12}
  M. Bessemoulin-Chatard and F. Filbet.
  \newblock A Finite Volume Scheme for Nonlinear Degenerate
  Parabolic Equations.
  \newblock {\em SIAM J. Sci. Comput.}, 34(5):B559–B583, 2012.

\bibitem{BiImKar11}
P.~Biler, C.~Imbert, and G.~Karch.
\newblock Barenblatt profiles for a nonlocal porous medium equation.
\newblock {\em C. R. Math. Acad. Sci. Paris}, 349(11-12):641--645, 2011.

\bibitem{BiImKa15}
P.~Biler, C.~Imbert, and G.~Karch.
\newblock The nonlocal porous medium equation: {B}arenblatt profiles and other
  weak solutions.
\newblock {\em Arch. Ration. Mech. Anal.}, 215(2):497--529, 2015.

\bibitem{BiKaMo10}
P.~Biler, G.~Karch, and R.~Monneau.
\newblock Nonlinear diffusion of dislocation density and self-similar
  solutions.
\newblock {\em Comm. Math. Phys.}, 294(1):145--168, 2010.

\bibitem{BuCaRo22}
L.~Bungert, J.~Calder, and T.~Roith.
\newblock Ratio convergence rates for euclidean first-passage percolation:
  Applications to the graph infinity laplacian.
\newblock {\em Ann. Appl. Probab.}, 34(4):3870--3910, 2024.

\bibitem{BuCaRo22a}
L.~Bungert, J.~Calder, and T.~Roith.
\newblock Uniform convergence rates for lipschitz learning on graphs.
\newblock {\em IMA J. Numer. Anal.}, 43(4):2445--2495, 2022.

\bibitem{CaSoVa13}
L.~Caffarelli, F.~Soria, and J.~L. V\'{a}zquez.
\newblock Regularity of solutions of the fractional porous medium flow.
\newblock {\em J. Eur. Math. Soc. (JEMS)}, 15(5):1701--1746, 2013.

\bibitem{CaVa11}
L.~Caffarelli and J.~L. Vazquez.
\newblock Nonlinear porous medium flow with fractional potential pressure.
\newblock {\em Arch. Ration. Mech. Anal.}, 202(2):537--565, 2011.

\bibitem{CaVa15}
L.~Caffarelli and J.~L. V\'{a}zquez.
\newblock Regularity of solutions of the fractional porous medium flow with
  exponent 1/2.
\newblock {\em Algebra i Analiz}, 27(3):125--156, 2015.

\bibitem{CaVa11b}
L.~A. Caffarelli and J.~L. V\'{a}zquez.
\newblock Asymptotic behaviour of a porous medium equation with fractional
  diffusion.
\newblock {\em Discrete Contin. Dyn. Syst.}, 29(4):1393--1404, 2011.


\bibitem{CCH15}
  J. A. Carrillo, A. Chertock, and Y. Huang.
  \newblock A Finite-Volume Method for Nonlinear Nonlocal
  Equations with a Gradient Flow Structure.
  \newblock {\em Commun. Comput. Phys.}, 17(01):233–258, 2015

\bibitem{CaFjSo21}
J.~A. Carrillo, U.~S. Fjordholm, and S.~Solem.
\newblock A second-order numerical method for the aggregation equations.
\newblock {\em Math. Comp.}, 90(327):103--139, 2021.


\bibitem{CaFrSu24}
J.~A.~Carrillo, S.~Fronzoni, and E.~Süli.
\newblock Finite element scheme for the fractional porous medium equation with
  fractional pressure.
\newblock {\em Numer. Math.}, 157:1537--1614, 2025.


\bibitem{CaFCVa22}
J.~A. Carrillo, D.~G\'{o}mez-Castro, and J.~L. V\'{a}zquez.
\newblock A fast regularisation of a {N}ewtonian vortex equation.
\newblock {\em Ann. Inst. H. Poincar\'{e} C Anal. Non Lin\'{e}aire},
  39(3):705--747, 2022.

\bibitem{ChJa17}
E.~Chasseigne and E.~R. Jakobsen.
\newblock On nonlocal quasilinear equations and their local limits.
\newblock {\em J. Differential Equations}, 262(6):3759--3804, 2017.

\bibitem{CP04}
  C. Chainais-Hillairet and Y.-J.Peng.
  \newblock Finite volume approximation for degenerate drift-diffusion
  system in several space dimensions.
  \newblock {\em Math. Models Methods Appl. Sci.} 14:461–481, 2004.


\bibitem{Cia-etal18}
O.~Ciaurri, L.~Roncal, P.~R. Stinga, J.~L. Torrea, and J.~L. Varona.
\newblock Nonlocal discrete diffusion equations and the fractional discrete
  {L}aplacian, regularity and applications.
\newblock {\em Adv. Math.}, 330:688--738, 2018.

\bibitem{CiJa11}
S.~Cifani and E.~R. Jakobsen.
\newblock Entropy solution theory for fractional degenerate
  convection-diffusion equations.
\newblock {\em Ann. Inst. H. Poincar\'{e} C Anal. Non Lin\'{e}aire},
  28(3):413--441, 2011.

\bibitem{CiJa14}
S.~Cifani and E.~R. Jakobsen.
\newblock On numerical methods and error estimates for degenerate fractional
  convection-diffusion equations.
\newblock {\em Numer. Math.}, 127(3):447--483, 2014.

\bibitem{CrLi96}
M.~G. Crandall and P.-L. Lions.
\newblock Convergent difference schemes for nonlinear parabolic equations and
  mean curvature motion.
\newblock {\em Numer. Math.}, 75(1):17--41, 1996.

\bibitem{dP-etal11}
A.~de~Pablo, F.~Quir\'{o}s, A.~Rodr\'{\i}guez, and J.~L. V\'{a}zquez.
\newblock A fractional porous medium equation.
\newblock {\em Adv. Math.}, 226(2):1378--1409, 2011.

\bibitem{dP-etal12}
A.~de~Pablo, F.~Quir\'{o}s, A.~Rodr\'{\i}guez, and J.~L. V\'{a}zquez.
\newblock A general fractional porous medium equation.
\newblock {\em Comm. Pure Appl. Math.}, 65(9):1242--1284, 2012.

\bibitem{De00}
K.~Deckelnick.
\newblock Error bounds for a difference scheme approximating viscosity
  solutions of mean curvature flow.
\newblock {\em Interfaces Free Bound.}, 2(2):117--142, 2000.

\bibitem{dTEnJa17}
F.~del Teso, J.~Endal, and E.~R. Jakobsen.
\newblock Uniqueness and properties of distributional solutions of nonlocal
  equations of porous medium type.
\newblock {\em Adv. Math.}, 305:78--143, 2017.

\bibitem{dTEnJa18}
F.~del Teso, J.~Endal, and E.~R. Jakobsen.
\newblock Robust numerical methods for nonlocal (and local) equations of porous
  medium type. {P}art {II}: {S}chemes and experiments.
\newblock {\em SIAM J. Numer. Anal.}, 56(6):3611--3647, 2018.

\bibitem{dTEnJa19}
F.~del Teso, J.~Endal, and E.~R. Jakobsen.
\newblock Robust numerical methods for nonlocal (and local) equations of porous
  medium type. {P}art {I}: {T}heory.
\newblock {\em SIAM J. Numer. Anal.}, 57(5):2266--2299, 2019.

\bibitem{dTLi22}
F.~del Teso and E.~Lindgren.
\newblock A finite difference method for the variational {$p$}-{L}aplacian.
\newblock {\em J. Sci. Comput.}, 90(1):Paper No. 67, 31, 2022.

\bibitem{dTMeOc23}
F.~del Teso, M.~Medina, and P.~Ochoa.
\newblock Higher-order asymptotic expansions and finite difference schemes for
  the fractional $p$-laplacian.
\newblock {\em Math. Ann.}, 390(1):157--203, 2024.

\bibitem{Dro10}
J.~Droniou.
\newblock A numerical method for fractal conservation laws.
\newblock {\em Math. Comp.}, 79(269):95--124, 2010.

\bibitem{DrJa14}
J.~Droniou and E.~R. Jakobsen.
\newblock A uniformly converging scheme for fractal conservation laws.
\newblock In {\em Finite volumes for complex applications {VII}. {M}ethods and
  theoretical aspects}, volume~77 of {\em Springer Proc. Math. Stat.}, pages
  237--245. Springer, Cham, 2014.


\bibitem{EGH00}
  R. Eymard, T. Gallouet, and R. Herbin.
  Finite volume methods. Handbook of numerical analysis, Vol. VII.
  \newblock North-Holland (2000), 713–1020. 
  
\bibitem{FeNe09}
X.~Feng and M.~Neilan.
\newblock Mixed finite element methods for the fully nonlinear
  {M}onge-{A}mp\`ere equation based on the vanishing moment method.
\newblock {\em SIAM J. Numer. Anal.}, 47(2):1226--1250, 2009.

\bibitem{Fjo-etall17}
U.~S. Fjordholm, R.~K\"{a}ppeli, S.~Mishra, and E.~Tadmor.
\newblock Construction of approximate entropy measure-valued solutions for
  hyperbolic systems of conservation laws.
\newblock {\em Found. Comput. Math.}, 17(3):763--827, 2017.


\bibitem{FoPaSc25}
G.~Foghem, D.~Padilla-Garza, and M.~Schmidtchen.
\newblock Gradient flow solutions for porous medium equations with nonlocal L\'evy-type pressure.
\newblock {\em Calc. Var. PDE}, 64(3):88, 2025.


\bibitem{FjSo16}
U.~S. Fjordholm and S.~Solem.
\newblock Second-order convergence of monotone schemes for conservation laws.
\newblock {\em SIAM J. Numer. Anal.}, 54(3):1920--1945, 2016.


\bibitem{FrOb11}
B.~D. Froese and A.~M. Oberman.
\newblock Convergent finite difference solvers for viscosity solutions of the
  elliptic {M}onge-{A}mp\`ere equation in dimensions two and higher.
\newblock {\em SIAM J. Numer. Anal.}, 49(4):1692--1714, 2011.

\bibitem{Gom23}
D.~G\'omez-Castro.
\newblock Beginner's guide to {A}ggregation-{D}iffusion {E}quations.
\newblock {\em SeMA J.: Bolet\'{\i}n de la Soc. Esp. de Matem\'atica Aplicada}, 81(4):531--587, 2024.


\bibitem{LiSa22}
W.~Li and A.~J. Salgado.
\newblock Two-scale methods for the normalized infinity laplacian: rates of convergence.
\newblock {\em IMA J. Numer. Anal.}, 44(5):2603--2666, 2023.

\bibitem{LiMaSe18}
S.~Lisini, E.~Mainini, and A.~Segatti.
\newblock A gradient flow approach to the porous medium equation with
  fractional pressure.
\newblock {\em Arch. Ration. Mech. Anal.}, 227(2):567--606, 2018.

\bibitem{LoRa05}
G.~Loeper and F.~Rapetti.
\newblock Numerical solution of the {M}onge-{A}mp\`ere equation by a {N}ewton's
  algorithm.
\newblock {\em C. R. Math. Acad. Sci. Paris}, 340(4):319--324, 2005.

\bibitem{mopost}
Mathoverflow.
\newblock Discontinuity of solutions to approximation schemes in the
  barles-souganidis framework.
\newblock
https://mathoverflow.net/questions/253954/discontinuity-of-solutions-to-approximation-schemes-in-the-barles-souganidis-fra.
\bibitem{NeTa92}
H.~Nessyahu and E.~Tadmor.
\newblock The convergence rate of approximate solutions for nonlinear scalar
  conservation laws.
\newblock {\em SIAM J. Numer. Anal.}, 29(6):1505--1519, 1992.

\bibitem{NoNtZh19}
R.~H. Nochetto, D.~Ntogkas, and W.~Zhang.
\newblock Two-scale method for the {M}onge-{A}mp\`ere equation: pointwise error
  estimates.
\newblock {\em IMA J. Numer. Anal.}, 39(3):1085--1109, 2019.

\bibitem{Ob04}
A.~M. Oberman.
\newblock A convergent monotone difference scheme for motion of level sets by
  mean curvature.
\newblock {\em Numer. Math.}, 99(2):365--379, 2004.

\bibitem{Obe05}
A.~M. Oberman.
\newblock A convergent difference scheme for the infinity {L}aplacian:
  construction of absolutely minimizing {L}ipschitz extensions.
\newblock {\em Math. Comp.}, 74(251):1217--1230, 2005.

\bibitem{Ob08}
A.~M. Oberman.
\newblock Wide stencil finite difference schemes for the elliptic
  {M}onge-{A}mp\`ere equation and functions of the eigenvalues of the
  {H}essian.
\newblock {\em Discrete Contin. Dyn. Syst. Ser. B}, 10(1):221--238, 2008.

\bibitem{Obe13}
A.~M. Oberman.
\newblock Finite difference methods for the infinity {L}aplace and
  {$p$}-{L}aplace equations.
\newblock {\em J. Comput. Appl. Math.}, 254:65--80, 2013.

\bibitem{RuSaSo19}
A.~M. Ruf, E.~Sande, and S.~Solem.
\newblock The optimal convergence rate of monotone schemes for conservation
  laws in the {W}asserstein distance.
\newblock {\em J. Sci. Comput.}, 80(3):1764--1776, 2019.

\bibitem{StTeVa16}
D.~Stan, F.~del Teso, and J.~L. V\'{a}zquez.
\newblock Finite and infinite speed of propagation for porous medium equations
  with nonlocal pressure.
\newblock {\em J. Differential Equations}, 260(2):1154--1199, 2016.

\bibitem{StTeVa18b}
D.~Stan, F.~del Teso, and J.~L. V\'{a}zquez.
\newblock Porous medium equation with nonlocal pressure.
\newblock In {\em Current research in nonlinear analysis}, volume 135 of {\em
  Springer Optim. Appl.}, pages 277--308. Springer, Cham, 2018.

\bibitem{StdTVa18}
D.~Stan, F.~del Teso, and J.~L. V\'{a}zquez.
\newblock Existence of weak solutions for a general porous medium equation with
  nonlocal pressure.
\newblock {\em Arch. Ration. Mech. Anal.}, 233(1):451--496, 2019.

\end{thebibliography}

\end{document}